\documentclass[11pt]{article}
\usepackage{amssymb,amsmath,amsthm,hyperref,graphicx,color,array,fullpage,mathrsfs, stmaryrd}
\usepackage[notref,notcite,color,final]{showkeys}
\usepackage{makeidx}
\input xy
\xyoption{all}

\makeindex

\newtheorem{theorem}{Theorem}[section]
\newtheorem{lemma}[theorem]{Lemma}
\newtheorem{corollary}[theorem]{Corollary}

\newtheorem{definition}[theorem]{Definition}

\newtheorem{remark}[theorem]{Remark}
\newtheorem{proposition}[theorem]{Proposition}

\author{Ellen Eischen\thanks{Ellen Eischen's research is partially supported by  National Science Foundation Grant DMS-1249384.  She would also like to thank Columbia University for its hospitality while she was a Visiting Scholar there during the spring of 2014.}\hspace{0.05in} and Xin Wan}

\title{\textsc{$p$-adic Eisenstein Series and $L$-Functions of certain cusp forms on definite unitary groups}}

\date{\today}

\newcommand{\adeles}{\mathbb{A}}

\newcommand{\CK}{\mathcal{K}}
\newcommand{\IC}{\mathbb{C}}
\newcommand{\IQ}{\mathbb{Q}}
\newcommand{\Hom}{\mathrm{Hom}}
\newcommand{\sph}{\mathrm{sph}}
\newcommand{\ZZ}{\mathbb{Z}}
\newcommand{\gl}{\mathrm{GL}}
\newcommand{\uo}{\underline{\omega}}
\newcommand{\uk}{\underline{k}}
\newcommand{\can}{\mathrm{can}}
\newcommand{\Kling}{\mathrm{Kling}}
\newcommand{\sieg}{\mathrm{Sieg}}
\newcommand{\smfin}{\mathrm{sm-fin}}
\newcommand{\Spec}{\mathrm{Spec}}
\newcommand{\low}{\mathrm{low}}

\newcommand{\GL}{\mathrm{GL}}

\begin{document}
\bibliographystyle{amsalpha}

\maketitle
\begin{abstract}
We construct $p$-adic families of Klingen Eisenstein series and $L$-functions for cuspforms (not necessarily ordinary) unramified at an odd prime $p$ on definite unitary groups of signature $(r, 0)$ (for any positive integer $r$) for a quadratic imaginary field $\mathcal{K}$ split at $p$.
When $r=2$, we show that the constant term of the Klingen Eisenstein family is divisible by a certain $p$-adic $L$-function.
\end{abstract}

\setcounter{tocdepth}{2}
\tableofcontents
\setcounter{secnumdepth}{3}

\section{Introduction}\label{intro-section}

\subsection{Brief overview and relationship with other work}

In this paper, we construct $p$-adic families of Klingen Eisenstein series for cusp forms of level prime to $p$ (which are unramified principle series representations at $p$ and hence of finite slope) on definite unitary groups.  When the signature of the unitary group is $(2, 0)$, we also show that the constant terms of these families of Klingen Eisenstein series are divisible by a certain $p$-adic $L$-function.  One motivation for doing this is that this is a step toward proving the Iwasawa Main Conjectures for not necessarily ordinary forms.  Note that unlike in the ordinary case, we need to construct vector-valued Eisenstein series.  (To construct non-ordinary families in the most generality, we need to include representations that are not necessarily one-dimensional, hence we must consider vector-valued -- not just scalar-valued -- cusp forms.)

This work builds on the constructions in the authors' prior papers \cite{Eischen, apptoHLS, apptoHLSvv, WAN}.    Unlike the work in \cite{apptoHLS, apptoHLSvv}, which was restricted to Siegel Eisenstein series, this paper also handles the case of Klingen Eisenstein series and pullbacks.  We also note that we expect that the $p$-adic $L$-function in this paper is a special case of the one that Harris, Li, and Skinner announced nine years ago in \cite{HLSpaper} that they were constructing in a work in progress; their anticipated result, however, is more general.  (We also note that the first author is collaborating with Harris, Li, and Skinner on the construction of these more general $p$-adic $L$-functions.)  Both the construction in this paper and the one announced in \cite{HLSpaper} rely on pullback integrals (e.g. See Section \ref{pbints}) similar to those arising in the doubling method \cite{GPSR}; thus the approach is similar for both projects.  The paper with Harris, Li, and Skinner, however, concerns only ordinary forms, while the present paper considers the more general finite slope case.  On the other hand, one of the aims of the work with Harris, Li, and Skinner is to obtain a more general interpolation formula than the one in the present paper.

\subsection{Structure of the paper}

In Section \ref{conventions-section}, we introduce some of the conventions with which we work.  We conclude Section \ref{conventions-section} by making use of some of these conventions in order to give a detailed statement of the main theorem.

In Section \ref{Background-section}, we give the background and definitions necessary for our discussion.  In particular, we introduce the unitary groups with which we will work.  We then review the basic theory of automorphic forms on unitary groups, and we provide the key facts about $p$-adic automorphic forms needed for our main results.  We also summarize some key results about differential operators from \cite{Eischen} and results on $q$-expansions, which are needed to prove our main results on $p$-adic interpolation.

In Sections \ref{EseriesPullbkFormulas-section} and \ref{local-computations-section}, we discuss both Siegel and Klingen Eisenstein series on unitary groups.  We first recall certain pullback formulas and results on Fourier coefficients of Eisenstein series.  We then choose specific local data that determine the Eisenstein series, and we compute the corresponding local integrals.

In Section \ref{Global-Computations-section}, we finish constructing the $p$-adic families of Eisenstein series, and we state the main results of this paper.  We first construct a $p$-adic Seigel Eisenstein measure.  Then, from this measure, we construct $p$-adic families of Klingen Eisenstein series.  In Theorem \ref{thisprop-proposition}, we show that two particular pullbacks are equal to each other, when the unitary group is of signature $(2,0)$.  As a consequence, we obtain Corollary \ref{cor5.9}, which says that when the signature of the unitary group is $(2, 0)$, the constant terms of a certain $p$-adic family of Eisenstein series that  we construct is divisible by a particular $p$-adic $L$-function.

\begin{remark}
For most of the material in the first three sections, we work with general unitary groups of arbitrary signature $(r, s)$.  In Sections \ref{FJexpns-section} and \ref{2.12}, though, we restrict to the case in which $s=1$.  The reason for the restriction in that portion of the paper is that we have a good reference in \cite{Hsieh} for the case $s=1$.  (While \cite{LAN}, our reference for much of the geometry, has no such restriction on the signature, it requires substantially more work and notation to formulate this case in our setting.)
\end{remark}

\subsection{Conventions and more detailed statement of main theorem}\label{conventions-section}

In this section, we introduce some of our conventions.

Let $p$ be an odd prime. Let $\mathcal{K}$ be an imaginary quadratic field in which $p$ splits as $v_0\bar{v}_0$.\footnote{Note that all of the results in this paper extend easily to the case of $\mathcal{K}$ a CM field, as in \cite{apptoHLS}.  Our requirement that $\mathcal{K}$ be an imaginary quadratic field is made purely to avoid the more cumbersome notation that results from working with a more general CM extension.}  By class field theory there is a unique $\mathbb{Z}_p^2$-extension $\mathcal{K}_\infty$ of $\mathcal{K}$ unramified away from $p$.   We fix an embedding $\sigma: \mathcal{K}\hookrightarrow \IC$, and we identify each element of $\mathcal{K}$ with its image under $\sigma$.  We  also fix an isomorphism $\iota:\mathbb{C}_p\simeq \mathbb{C}$ such that $v_0$ is induced by this isomorphism.  Note that this fixes a choice of a CM type $\Sigma_{CM}$.  We let $\Gamma_\mathcal{K}=\mathrm{Gal}(\mathcal{K}_\infty/\mathcal{K})$. Complex conjugation $c$ gives an involution on $\Gamma_\mathcal{K}$. We let $\Gamma_{\CK}^{\pm}\subset \Gamma_\mathcal{K}$ be the rank one $\mathbb{Z}_p$-submodule on which $c$ acts by $\pm 1$.

Let $\adeles$ denote the adeles over $\IQ$, and for any field $F$, let $\adeles_F$ denote the adeles over $F$.  Let $\adeles_f$ denote the adeles at the finite places.  For any number field $F$, we denote by $\mathcal{O}_F$ the ring of integers in $F$.

Let $\pi$ be an irreducible cuspidal automorphic representation of the definite general unitary group $GU(r,0)$ of signature $(r, 0)$ and weight $\underline{k}=(a_1,a_2,\ldots,a_r)$ with $a_1\geq\cdots\geq a_r\geq 0$, and let $\varphi\in\pi$.  Then  $\varphi$ can be viewed as a function on a finite set of points.  Let $\xi_0$ be a Hecke character on $\mathcal{K}^\times\backslash \mathbb{A}_\mathcal{K}^\times$ such that $\xi_0|\cdot|^{-\frac{r-1}{2}}$ is of finite order.
 Let $L$ be a finite extension of $\mathbb{Q}_p$ containing all the values of $\varphi$ and $\xi_0|\cdot|^{-\frac{r-1}{2}}$.     Let $\Sigma$ be a finite set of primes of $\IQ$ containing the primes above $p$ and all the primes at which $\xi_0$, $\pi$, or $\mathcal{K}/\mathbb{Q}$ is ramified.  Since $\mathcal{K}_p\simeq \mathcal{K}_{v_0}\times\mathcal{K}_{\bar{v}_0}\simeq \mathbb{Q}_p\times\mathbb{Q}_p$, we have
$$GU(r,0)(\mathbb{Q}_p)\simeq \mathrm{GL}_r(\mathbb{Q}_p)\times\mathbb{Q}_p^\times, g\mapsto (g_{v_0},\mu(g))$$
where $\mu$ is the similitude character.

We define $\Lambda:=\Lambda_\mathcal{K}:=\mathbb{Z}_p\llbracket\Gamma_\mathcal{K}\rrbracket$. For any finite extension $A$ of $\mathbb{Z}_p$, define $\Lambda_{\mathcal{K},A}:=A\llbracket\Gamma_\mathcal{K}\rrbracket$.  In particular, $\Lambda_{\mathcal{K},\mathcal{O}_L}:=\mathcal{O}_L\llbracket\Gamma_\mathcal{K}\rrbracket$, $\Lambda_{\mathcal{K},\mathcal{O}_L}^+:=\mathcal{O}_L[\![\Gamma_\mathcal{K}^+]\!]$, and $\Lambda_{\mathcal{K},\mathcal{O}_L}^-:=\mathcal{O}_L[\![\Gamma_\mathcal{K}^-]\!]$.  In Section \ref{Global-Computations-section}, we will define families of Hecke characters $\tau_\phi$, $\xi_\phi$ parameterized by elements $\phi\in\mathrm{Spec}\Lambda_{\mathcal{K},\mathcal{O}_L}$.

  We now introduce some notation that we will use when working with matrix groups.  We write $\mathrm{diag}(g, h)$ to mean a block diagonal matrix with $g$ in the upper left hand corner and $h$ in the lower right hand corner, and we denote the space of $r\times r$ matrices over a ring $R$ by $M_r(R)$.  Denote by $S_n(R)$ the space of $n\times n$ Hermitian matrices over a ring $R$, and let $S_n^+(R)$ denote the subset of $S(R)$ consisting of positive definite matrices.  Let $K=\prod_v' K_v$ be a compact open subgroup of $GU(r,0)(\mathbb{A}_f)$ such that $K_p=\mathrm{GL}_r(\mathbb{Z}_p)\times \mathbb{Z}_p^\times$.  Fix a Borel subgroup $B$ in $GU(r, 0)$, and let $N$ denote the unipotent radical of $B$.  We let $\Gamma_0(p^t)\subseteq GU(r,0)(\mathbb{Z}_p)$ be the subgroup consisting of matrices congruent to $B(\mathbb{Z}_p)$ modulo $p^t$, and we let $\Gamma_1(p^t)\subseteq \Gamma_0(p^t)$ be the subgroup consisting of matrices congruent to $N(\mathbb{Z}_p)$ modulo $p^t$. Let $K_0(p)=\prod_{v\not=p} K_v \prod_{v=p}\Gamma_0(p)$.

  \subsubsection{The main theorem}

The main results of the paper appear in Section \ref{constterms-section}.  In Theorem \ref{main}, we list these results.

\begin{theorem}\label{main}
Let $\pi$ be a tempered irreducible cuspidal automorphic representation of $GU(r,0)$ of weight $\underline{k}=(a_1,a_2,\ldots,a_r)$, $a_1\geq \cdots \geq a_r\geq 0$, such that $\pi_p$ is the unramified principal series representation $\pi(\chi_1,\ldots,\chi_r)$ for characters $\chi_1, \ldots, \chi_r$ such that the $\chi_i$'s are pairwise distinct.  Let $\varphi\in\pi^{K_0(p)}$ be an eigenvector for all the Hecke operators at $p$.
\begin{itemize}
\item[(i)] There is a constant $C_{\varphi,p}$ and an element $\mathcal{L}_{\varphi,\xi_0}^\Sigma\in\Lambda_{\mathcal{K},\mathcal{O}_L}$ such that for a certain Zariski dense set of arithmetic points $\phi\in\mathrm{Spec}\Lambda_{\mathcal{K},\mathcal{O}_L}$ (to be specified in the text) we have
\begin{align*}
\phi(\mathcal{L}_{\varphi,\xi_0}^\Sigma)=C_{\varphi,p}\cdot c_{\phi, \pi}\cdot L^\Sigma(\tilde{\pi},\xi_\phi,0)
\end{align*}
where $L^\Sigma(\tilde{\pi}, \xi_\phi, 0)$ is the $L$-function (with the factors at $\Sigma$ removed) associated to the contragredient $\tilde{\pi}$ and a certain Hecke character $\xi_\phi$ dependent on $\phi$, and $c_{\phi, \pi}$ is a parameter dependent on $\phi$ and $\pi$ (which we make precise once we have introduced more notation).

\item[(ii)] There is a set of formal $q$-expansions\footnote{We review the key features of Fourier-Jacobi expansions and $q$-expansions in Section \ref{FJexpns-section}.} $\mathbf{E}_{\varphi,\xi_0}:=\{\sum_\beta a_{[g]}^t(\beta)q^\beta\}_{([g],t)}$ with $\sum_\beta a_{[g]}^t(\beta)q^\beta\in\Lambda_{\mathcal{K},\mathcal{O}_L}\otimes_{\mathbb{Z}_p}\mathcal{R}_{[g],\infty}$ where $\mathcal{R}_{[g],\infty}$ is a certain ring defined in Section \ref{algthy-section} and $([g],t)$ are \it{ $p$-adic cusp labels}, such that for a Zariski dense set of arithmetic points $\phi\in\mathrm{Spec}\Lambda_{\mathcal{K},\mathcal{O}_L}$, $\phi(\mathbf{E}_{\varphi,\xi_0})$ is the Fourier-Jacobi expansion of the (projection onto the) highest weight vector of a holomorphic Klingen Eisenstein series $E(f_{\Kling,\phi},z_{\kappa_\phi},-)$, which is an eigenvector for the Hecke operator $U_{t^+}$ (introduced in Section \ref{hecke-section}) with non-zero eigenvalue. Here, $f_{\Kling, \phi}$ is a certain \it{Klingen section} to be defined in the text.

\item[(iii)] If $r=2$, then the constant terms $a_{[g]}^t(0)$ in the $q$-expansion $\mathbf{E}_{\varphi, \xi_0}$ are divisible by $\mathcal{L}_{\varphi,\xi_0}^\Sigma\cdot\mathcal{L}_{\bar{\tau}'}^\Sigma$ where $\mathcal{L}_{\bar{\tau}'}^\Sigma$ is the $p$-adic $L$-function of a Dirichlet character (originally constructed by Kubota and Leopoldt in \cite{KL}).
\end{itemize}
\end{theorem}
\begin{remark}
The Fourier-Jacobi expansion in (ii) is obtained from a bounded measure of $p$-adic modular forms on $GU(r+1,1)$ on $\Gamma_\mathcal{K}\simeq \mathbb{Z}_p^2$, interpolating the corresponding holomorphic Klingen Eisenstein series.
\end{remark}
\begin{remark}
In part (iii), we assumed $r=2$. This is because we need to know that certain constants ($c_{(\underline{k},0,\kappa)}$ and $c_{(\underline{k},0,\kappa)}'$) coming from different Archimedean integrals are equal. We do not compute the precise formulas for them. Instead we show this by comparison with the construction in the ordinary case.  In fact, all what we need is the existence of some split prime $\ell$ such that there is an irreducible cuspidal automorphic representation of $GU(r,0)$ that is ordinary at the prime $\ell$ and is tempered.  If $r=2$, this follows from the Sato-Tate Conjecture. There are other cases where this can be done. For example, when $r$ is even, one might construct such automorphic representations by inducing certain Hecke characters of CM fields.  We leave this to the reader.
\end{remark}

\section{Background}\label{Background-section}
\subsection{Notation}\label{notation-section}
 Let $\epsilon$ denote the cyclotomic character and $\omega$ the Techimuller character. For a character $\psi$ of $\mathcal{K}_v^\times$ or $\mathbb{A}_\mathcal{K}^\times$ we often write $\psi'$ for the restriction to $\mathbb{Q}_v^\times$ or $\mathbb{A}_\mathbb{Q}^\times$. For a character $\tau$ of $\mathcal{K}^\times$ or $\mathbb{A}_\mathcal{K}^\times$ we define $\tau^c$ by $\tau^c(x)=\tau(x^c)$, where $x^c$ denotes the complex conjugate of $x$.

Let $\xi_p$ be a character of $\mathbb{Q}_p$.  If the conductor is $(p^t)$ with $t>0$, then we define the Gauss sum $\mathfrak{g}(\xi_p)=\sum_{a\in(\mathbb{Z}_p/p^t\mathbb{Z}_p)^\times}\xi_p(a)e^{2\pi ia/p^t}$.  Later, we will also view the Gauss sum as a function on multiplicative characters in the usual way.

 Let $\Psi={\Psi}_\mathcal{K}:G_\mathcal{K}\rightarrow \Gamma_\mathcal{K}\hookrightarrow\Lambda_\mathcal{K}^\times$ be the canonical character. We define $\boldsymbol{\varepsilon}_\mathcal{K}$ to be the composition of $\Psi_\mathcal{K}$ with the reciprocity map of global class field theory, which we denote by $\mathrm{rec}_\mathcal{K}$.  Here we use the geometric normalization of class field theory. We also define $\Psi^{\pm}$ to be the composition of $\Psi$ with the $\mathcal{O}_L$-morphism $\Lambda_{\mathcal{K},\mathcal{O}_L}\rightarrow \Lambda_{\mathcal{K},\mathcal{O}_L}^{\pm}$ taking $\gamma^\mp$ to $0$, where $\gamma^\pm$ denotes a topological generator of $\Lambda_{\mathcal{K},\mathcal{O}_L}^{\pm}$.

  For $n>0$, we write $B=B_n$ for the standard upper triangular Borel subgroup of $\GL_n$ and $N=N_n$ for its unipotent radical. We write ${ }^tB$ and ${ }^tN$ for the opposite Borel and unipotent radical. For a representation $V_1$ of $G_1$ and $V_2$ of $G_2$, we write $V_1\boxtimes V_2$ for the natural representation of $G_1\times G_2$ on $V_1\otimes V_2$. We use this notation to distinguish from the tensor product representation of a single group, which we denote by $\otimes$.

  We refer to \cite[Section 2.8]{Hsieh} for the discussion of the CM period $\Omega_\infty$ and the $p$-adic period $\Omega_p$ associated to $\mathcal{K}$.

\subsection{Unitary Groups}\label{unitarygroups-section}
Let  $\zeta$ be a diagonal matrix such that $i^{-1}\zeta$ is positive definite with entries in $\mathcal{O}_{\mathcal{K}}$ prime to $p$, and let $GU(r,0)$ denote the unitary similitude group associated to $\zeta$.  Let $GU(r+1,1)$ denote the unitary similitude group associated to the matrix $\begin{pmatrix}&&1\\&\zeta&\\-1&&\end{pmatrix}$.   Note that the signature of $GU(r, 0)$ (respectively, $GU(r+1, 1)$) is $(r, 0)$ (respectively, $(r+1, 1)$).
More generally, given nonnegative integers $r\geq s$, we define
\begin{align*}
\theta_{r,s}=\begin{pmatrix}&&1_s\\&\zeta&\\-1_s&&\end{pmatrix}
\end{align*}
\index{$\theta_{r,s}$} where $\zeta$ is a diagonal $(r-s)\times (r-s)$ matrix such that $i^{-1}\zeta$ is positive definite.  (When $r = s$, we define $\theta_{r,s}=\begin{pmatrix}&1_s\\-1_s&\end{pmatrix}.$)  Let $V=V(r,s)$ be the skew Hermitian space over $\mathcal{K}$ with respect to this metric, i.e. $\mathcal{K}^{r+s}$ equipped with the metric given by
\begin{align}\label{pairingrs-def}
\langle u,v\rangle:=u\theta_{r,s}{}^t\!\bar{v}.
\end{align} We define algebraic groups $G:=GU(r,s)$ \index{$G$} and $U(r,s)$ whose $R$-points, for any $\IQ$-algebra $R$, are
\begin{align}\label{GU-defn}
G(R)=GU(r,s)(R):=\left\{g\in GL_{r+s}(\mathcal{K}\otimes_\IQ R)|g\theta_{r,s}g^*=\mu(g)\theta_{r,s},\mu(g)\in R^\times\right\},
\end{align}
where $g^*:={ }^t\bar{g}$,
and
$$U(r,s)(R):=\{g\in GU(r,s)(R)|\mu(g)=1\}.$$
(The function $\mu: GU(r,s)\rightarrow \mathbb{G}_m$ is called the {\it similitude character}.)
Given a positive integer $n$, we sometimes write $GU_n$ and $U_n$ for $GU(n,n)$ and $U(n,n)$, respectively.
We have the following embedding:
\begin{align}\label{r1-embedding}
GU(r,0)\times \mathrm{Res}_{\mathcal{K}/\mathbb{Z}}\mathbb{G}_m&\rightarrow GU(r+1,1)\\
g\times x&\mapsto m(g, x):= \begin{pmatrix}\mu(g)\bar{x}^{-1}&&\\&g&\\&&x\end{pmatrix}\nonumber.
\end{align}
Let $P$ be the parabolic subgroup of $GU(r+1,1)$ consisting of matrices such that the entries in the first column below the diagonal and the entries in the last row to the left of the diagonal are $0$. We let $M_P$ be the Levi subgroup of $P$.  We define $G_P\subseteq M_P$ to be the set of block diagonal matrices $\mathrm{diag}(1, g, \mu g)$ with $g\in GU(r, 0)$.  For $r\geq s$, we define $U(s, r) = U(r, s)$, viewed as a unitary group with the opposite signature of the group $U(r, s)$ defined above.  We define an embedding
$$\gamma: \{g_1\times g_2\in GU(r+1,1)\times GU(0,r), \mu(g_1)=\mu(g_2)\}\rightarrow GU(r+1,r+1)$$
$$g_1\times g_2\rightarrow S^{-1}\mathrm{diag}(g_1,g_2)S$$
for
\begin{align}\label{Sdefn-equ}
S=\begin{pmatrix}1&&&\\&1&&\frac{\zeta}{2}\\&&1&\\&-1&&-\frac{\zeta}{2}\end{pmatrix}.
\end{align}
  We also define embeddings
$$\gamma': \{g_1\times g_2\in GU(r,0)\times  GU(0,r), \mu(g_1)=\mu(g_2)\}\rightarrow GU(r,r)$$
$$g_1\times g_2\rightarrow {S'}^{-1}\mathrm{diag}(g_1,g_2)S'$$
for
\begin{align}\label{Sprimedefn-equ}
S'=\begin{pmatrix}1&-\frac{\zeta}{2}\\-1&-\frac{\zeta}{2}\end{pmatrix}.
\end{align}
\begin{remark}
We work with coordinates here, because they are useful for our later computations.  We note, though, that one could rephrase the discussion of this section in a coordinate-free way (at the expense of not having already chosen coordinates for our later computations).  More specifically, if $\langle, \rangle_{r, s}$ is the pairing on $V(r, s)$ defined in Equation \eqref{pairingrs-def} and $\langle, \rangle_{r, r}$ is the pairing on $V(r,r) = V(r, 0)\oplus V(r, 0)$ defined by $\langle(v, w), (v', w')\rangle_{r,r} = \langle v, v'\rangle_{r, 0}-\langle w, w'\rangle_{r, 0}$, then the natural embedding of unitary groups $U(\langle, \rangle_{r, 0})\times U(-\langle , \rangle_{r, 0})\hookrightarrow U(\langle, \rangle_{r, r})$ preserving these pairings is the same as the embedding $\gamma'$.  Similarly, if we write
\begin{align*}
\langle, \rangle_{r+1, r+1} = \langle, \rangle_{r+1, 1}\oplus \langle,\rangle_{0, r}
\end{align*}
(following the notation of Shimura in \cite[Section 1.1]{Shi97}),
then $\gamma$ is the natural embedding
\begin{align*}
U(\langle, \rangle_{r+1, 1})\times U(-\langle, \rangle_{r, 0})\hookrightarrow U(\langle, \rangle_{r+1, r+1}).
\end{align*}
\end{remark}

\subsection{Hermitian Symmetric Domain}\label{domain}
Suppose $r\geq s>0$. When there is no ambiguity about $r$ and $s$, we shall write $\theta$ in place of $\theta_{r,  s}$.  Then the Hermitian symmetric domain for $GU(r,s)$ is
$$X^+=X_{r,s}=\left\{\tau=\begin{pmatrix}x\\y \end{pmatrix}|x\in M_s(\mathbb{C}),y\in M_{(r-s)\times s}(\mathbb{C}),i(x^*-x)>-iy^*\theta^{-1}y\right\}.$$\index{$X^+$}   Note that when $r =s$, we take $X_{r, s}$ to be
\begin{align*}
X_{r, r} = \left\{x\in M_r\left(\mathbb{C}\right)|i(x^*-x)>0\right\}.
\end{align*}
For $\alpha\in GU(r,s)(\mathbb{R}),$ we write
$$\alpha=\begin{pmatrix}a&b&c\\g&e&f\\h&l&d\end{pmatrix}$$
according to the standard basis of $V$ together with the block decomposition with respect to $s+(r-s)+s$. There is an action of $\alpha\in G(\mathbb{R})^+$  (Here, the superscript $+$ denotes the component with positive similitude at the Archimedean place.) on $X_{r,s}$ defined by
$$\alpha\begin{pmatrix}x\\y\end{pmatrix}=\begin{pmatrix}ax+by+c \\gx+ey+f \end{pmatrix}(hx+ly+d)^{-1}.$$
If $rs=0$, $X_{r,s}$ consists of a single point written $\boldsymbol{x}_0$ with the trivial action of $G$. For an open compact subgroup $U$ of $G(\mathbb{A}_{\mathbb{Q},f})$, put
$$M_G(X^+,U):=G(\mathbb{Q})^+\backslash \left(X^+\times G(\mathbb{A}_{\mathbb{Q},f})\right)/U$$ where $U$ is an open compact subgroup of $G(\mathbb{A}_{\mathbb{Q},f})$.  (The equivalence is via $(gx, ghu)\sim (x, h) $ for all $g\in G(\IQ)^+$, $x\in X^+$, $h\in G\left(\adeles_{\IQ, f}\right)$, and $u\in U$.)

Let $\bf{i'}$ and $\bf{i''}$ be the points on the Hermitian symmetric domains for $GU(r,s)$ and $GU(r+1,s+1)$, respectively, defined by $\bf{\tilde{i}'}=\begin{pmatrix}i1_s\\0\end{pmatrix}$ and $\bf{\tilde{i}} = \begin{pmatrix}i1_{s+1}\\0\end{pmatrix}$. (Here, $0$ denotes the $(r-s)\times s$ or $(r-s)\times (s+1)$ $0$-matrix.) Let $GU(r,s)(\mathbb{R})^+$ be the subgroup of $GU(r,s)(\mathbb{R})$ whose similitude factor is positive.  Let $K_\infty^+$ and $K_\infty^{+,'}$ be the compact subgroups of $U(r+1,s+1)(\mathbb{R})$ and $U(r,s)(\mathbb{R})$, respectively, stabilizing $\bf{\tilde{i}}$ or $\bf{\tilde{i}'}$, respectively.  Let $K_\infty$  (resp. $K_{\infty}'$) \index{$K_{\infty}'$} be the groups generated by $K_\infty^+$ (resp. $K_\infty^{+,'}$)\index{$K_\infty^{+,'}$} and $\mathrm{diag}(1_{r+1},-1_{s+1})$ (resp. $\mathrm{diag}(1_r,-1_s)$).

  We also define an embedding of Hermitian domains
\begin{align}
X_{r+1,1}\times X_{0,r}&\hookrightarrow X_{r+1,r+1}\\
(z,x_0)&\hookrightarrow \begin{pmatrix}x&0\\y&\zeta\end{pmatrix}\label{rhstemp1}
\end{align}
for $z=\begin{pmatrix}x\\y\end{pmatrix}$. This embedding respects the actions of the unitary groups under the embedding $\gamma$.

\subsection{Automorphic forms}
We define a weight $\underline{k}=(a_1,\cdots,a_r;b_1,\ldots,b_s)$ for integers $a_1\geq\cdots\geq a_r\geq -b_1+r+s\geq\cdots\geq -b_s+r+s$.  (This convention follows \cite{Hsieh}.)  We define $||\underline{k}||=b_1+\cdots+b_s+a_{1}+\cdots+a_{r}$. By a scalar weight $\kappa$ for some $\kappa>r+1$ we mean the weight $(0,\ldots,0;\kappa,\ldots,\kappa)$.
With slight modifications, we will mainly follow \cite{Hsieh}, which in turn summarizes relevant portions of \cite{Shi97}, \cite{Shi00}, and \cite{Hida04}, to define the space of automorphic forms.
We define a cocycle $J: R_{F/\mathbb{Q}} G(\mathbb{R})^+\times X^+\rightarrow GL_r(\mathbb{C})\times GL_s(\mathbb{C}):=H(\mathbb{C})$ \index{$H$} by
$J(\alpha,\tau)=(\kappa(\alpha,\tau),\mu(\alpha,\tau))$, where for $\tau=\begin{pmatrix}x \\y \end{pmatrix}$ and
$\alpha=\begin{pmatrix}a&b&c\\g&e&f\\h&l&d\end{pmatrix}$,
$$\kappa(\alpha,\tau)=\begin{pmatrix}\bar{h}{}^t\!x+\bar{d}&\bar{h}{}^t\!y+l\bar{\theta}\\-\bar{\theta}^{-1}(\bar{g}{}^t\!x+\bar{f})&-\bar{\theta}^{-1}
\bar {g}{}^t\!y+\bar{\theta}^{-1}\bar{e}\bar{\theta}\end{pmatrix},\ \mu(\alpha,\tau)=hx+ly+d.$$

As in \cite{Hsieh}, we define some rational representations of $GL_r$. Let $R$ be a $\mathbb{Z}$-algebra. For a weight $\underline{k}=(a_1,\cdots,a_r;b_1,\ldots,b_s)$, we define the representation $L_{\underline{k}}(R)$ with minimal weight $-\underline{k}$ by to be the $R$-points of
\begin{align}\label{Lkdef}
L_{\underline{k}}=\{f\in\mathcal{O}_{\GL_r\times \GL_s}|f(tn_+g)=k^{-1}(t)f(g),t\in T_r\times T_s,n_+\in N_r\times {}^t\!N_s\},
\end{align}
where $T_r$ and $T_s$ denote maximal tori inside of the Borel subgroups $B_r$ and $B_s$, respectively.  The action $\rho_{\uk}$ on $L_{\underline{k}}$ is given by $\rho_{\uk}(g)(h) = f(hg).$  We define the functional $l_{\underline{k}}$ on $L_{\underline{k}}$ by evaluating at the identity. We also define the representation $L^{\underline{k}}(R)$ with highest weight $\underline{k}$ which has the same space as $L_{\underline{k}}$ but with the group action of $g\in GL_r$ given by $\rho^{\uk}(g) = \rho_{\underline{k}}({}^t\!g^{-1})$.  Note that $L_{\underline{k}}$ and $L^{\underline{k}}$ are dual to each other; we denote the natural pairing between them by $\langle, \rangle$.

\begin{definition}\label{odef}
Let $K$ be an open compact subgroup in $G(\mathbb{A}_{F,f})$.  The space $M_{\underline{k}}(K,\mathbb{C})$ of holomorphic modular forms of weight $\underline{k}$ is the space of holomorphic $L^{\underline{k}}(\mathbb{C})$-valued functions $f$ on $X^+\times G(\mathbb{A}_{\mathbb{Q},f})$ such that for all $\tau\in X^+$, $\alpha\in G(\mathbb{Q})^+$ and $u\in K$,
$$
f(\alpha\tau,\alpha gu)=\mu(\alpha)^{-\|\underline{k}\|}\rho^{\underline{k}}(J(\alpha,\tau))f(\tau,g).$$  When $r = s =1$, we also require a moderate growth condition at the cusps.
\end{definition}
We sometimes write $M^{(r,s)}_{\underline{k}}$ to emphasize the signature $(r, s)$ of the unitary group on which we are working.

\subsection{Shimura varieties and Igusa varieties}\label{2.5}
\subsubsection{Unitary Similitude Groups}
To each open compact subgroup $K = \prod_v K_v$ of $GU(r,s)(\mathbb{A}_f)$ whose $p$-component $K_p$ is $GU(r,s)(\mathbb{Z}_p)$, we attach a Shimura variety $S_G(K)$; we refer to \cite{Hsieh} for the definitions and details about arithmetic models of Shimura varieties.  The space $S_G(K)$ parametrizes quadruples $(A,\lambda, \iota, \bar{\eta}^{(\Box)})_{/S}$ where $\Box$ is a finite set of primes, $A$ is an abelian variety over some base ring $S$, $\lambda$ is an orbit of prime-to-$\Box$ polarizations of $A$, $\iota$ is an embedding of $\mathcal{O}_\mathcal{K}$ into the endomorphism ring of $A$ such that the action of $\iota(a)$ on $\mathrm{Lie}{A}$ has characteristic polynomial $(X-\bar{a})^r(X-a)^s$, and $\bar{\eta}^{(\Box)}$ is a prime-to-$\Box$ level structure of $A$.  Fix a coefficient ring $R$ and a quadruple $(A,\lambda, \iota, \bar{\eta}^{(\Box)})_{/\mathrm{Spec}R}$. We let $\omega_A=\Hom_R(\mathrm{Lie}A,R),$ with the action of $\mathcal{O}_\mathcal{K}$ given by $(x\cdot f)(m)=f(\iota(\bar{x})m)$ for $x\in \mathcal{O}_\mathcal{K}$ and $m\in\mathrm{Lie}A$.  (Note that this convention of Hsieh is different from that used by Shimura and Hida in the literature; the group denoted by $U(r,s)$ by Hsieh is the group denoted by $U(s,r)$ in the work of Shimura, Hida, and others.)
There is also a theory of compactifications of $S_G(K)$ developed in \cite{LAN}. We denote $\bar{S}_G(K)$ the toroidal compactification and $S^*_G(K)$ the minimal compactification.

  We define level groups at $p$ as in \cite[Section 1.10]{Hsieh}. Recall that $K=\prod_{v}K_v$ is such that $K_p=G(\mathbb{Z}_p)$.  If we write $g_p=\begin{pmatrix}A&B\\C&D\end{pmatrix}$ with $A$ an $r\times r$ matrix and $D$ an $s\times s$ matrix, then we define
\begin{align*}
K^n&=\{g\in K|g_p\equiv \begin{pmatrix}1_r&*\\0&1_s\end{pmatrix}\mathrm{mod}p^n\}\\
K^n_1&=\{g\in K|A\in N_r(\mathbb{Z}_p)\mathrm{mod}p^n, B\in { }^tN_s(\mathbb{Z}_p)\mathrm{mod} p^n, C\equiv0\mathrm{mod} p^n\}\\
K^n_0&=\{g\in K|A\in B_r(\mathbb{Z}_p)\mathrm{mod}p^n, B\in { }^tB_s(\mathbb{Z}_p)\mathrm{mod} p^n, C\equiv0\mathrm{mod} p^n\}.
\end{align*}

  Now we recall briefly the notion of Igusa varieties in \cite[Section 2]{Hsieh}, which summarizes earlier work of H. Hida. Let $V$ be the Hermitian space for $U(r,s)$, let $M$ be a standard lattice of $V$, and let $M_p=M\otimes_{\mathbb{Z}}\mathbb{Z}_p$. Let $\mathrm{Pol}_p=\{N^{-1}, N^0\}$ be a polarization of $M_p$. Recall that this is a polarization if $N^{-1}$ and $N^0$ are maximal isotropic submodules in $M_p$ and they are dual to each other with respect to the Hermitian metric on $V$ and also that:
$$\mathrm{rank} N_{v_0}^{-1}=\mathrm{rank}N_{\bar{v}_o}^0=r, \mathrm{rank}N_{\bar{v}_0}^{-1}=\mathrm{rank}N_{v_0}^0=s.$$
The Igusa variety of level $p^n$ is the scheme representing the usual quadruple for a Shimura variety (or its compactification) together with an embedding
$$j:\mu_{p^n}\otimes_\mathbb{Z}N^0\hookrightarrow A[p^n]$$
where $A$ is the abelian variety in the quadruple we use to define the arithmetic model of the Shimura variety. Note that the existence of $j$ implies that if $p$ is nilpotent in the base ring, then $A$ must be ordinary.  This is a Galois covering of the ordinary locus of the Shimura variety with Galois group $\mathbf{H}\simeq \mathrm{GL}_r\left(\ZZ_p\right)\times\mathrm{GL}_s\left(\ZZ_p\right)$.
We also denote $I_G(K_1^n)$ and $I_G(K_0^n)$ the Igusa varieties (over the toroidal compactification) with the corresponding level structures.

\subsubsection{Igusa Schemes for Unitary Groups}\label{Igusaunitary}

  Assume the tame level group $K$ is neat. For any $c$ an element in $\mathbb{Q}_+\backslash \mathbb{A}_{\mathbb{Q},f}^\times \nu(K)$ we refer to \cite[Section 2.5]{Hsieh} for a discussion of Igusa schemes and the ``$c$-Igusa schemes'' for the unitary groups $U(r,s)$ (not the similitude group).  They parameterize quintuples $(A,\lambda,\iota,\bar{\eta}^{(p)},j)_{/S}$ similar to the Igusa schemes for unitary similitude groups but requiring $\lambda$ to be a prime to $p$-polarization of $A$. If we write $I_1(K_1^n)$, $I_2(K_1^n)$ and $I_3(K_1^n)$ for the Igusa schemes of $U(r,1)$ ($U(0,r)$), $U(0,r)$ and $U(r+1,r+1)$ ($U(r,r)$), then there is a map of Igusa schemes:
$$i: I_1(K_1^{n})\times I_2(K_1^{n})\rightarrow I_3(K_3^{n}).$$
We refer to \cite[Section 2.6]{Hsieh} for details of this map as well as the corresponding version for $c$-Igusa scheme versions. Later on we will uses them for algebraic definitions for the pullback formulas for unitary and unitary similitude groups.

\subsection{Geometric Modular Forms}\label{geomforms-section}

  Let $H=GL_r\times GL_s$.  (Note that $H$ denotes $GL_r\times GL_s$, while $\mathbf{H}$ denotes the Galois group in Section \ref{2.5} that is isomorphic to $H\left(\ZZ_p\right)$.)  We define $\underline{\omega}=e_*\Omega_{\mathcal{G}/\bar{S}_G(K)}$ (where $e: \mathcal{G}\rightarrow \bar{S}_G(K)$ is the universal abelian scheme).  So $\underline{\omega}=e_{v_0}\underline{\omega}\oplus e_{\bar{v}_0}\underline{\omega}$, where $e_{v_0}$ and $e_{\bar{v}_0}$ denote the projectors onto the submodules on which each element $\alpha\in\mathcal{K}$ acts as multiplication by $\alpha$ or multiplication by $\bar{\alpha}$, respectively. We also define
$$\mathcal{E}^+:=\underline{\mathrm{Isom}}(\mathcal{O}^r_{\bar{S}_G(K)},e_{v_0}\underline{\omega}),$$
$$\mathcal{E}^-:=\underline{\mathrm{Isom}}(\mathcal{O}^s_{\bar{S}_G(K)},e_{\bar{v_0}}\underline{\omega}).$$
This is a $H$-torsor over $\bar{S}_G(K)$. We can define the automorphic sheaf $\omega_{\underline{k}}=\mathcal{E}\times^HL_{\underline{k}}$. A section $f$ of $\underline{\omega}_{\underline{k}}$ is a morphism $f:\mathcal{E}\rightarrow L_{\underline{k}}$ such that
$$f(x,h\boldsymbol{\omega})=\rho_{\underline{k}}(h)f(x,\boldsymbol{\omega}), h\in H.$$

  Now we consider automorphic forms on unitary groups in the adelic language.  The space of automorphic forms of weight $\underline{k}$ and level $K$ with central character $\chi$ consists of smooth and slowly increasing functions $F: G(\mathbb{A}_\mathbb{Q})\rightarrow L_{\underline{k}}(\mathbb{C})$ such that for every $(\alpha,k_\infty,u,z)\in G(\mathbb{Q})\times K_\infty^+\times K\times Z(\mathbb{A}_\mathbb{Q})$,
$$F(z\alpha gk_\infty u)=\rho_{\underline{k}}(J(k_\infty,\tilde{\bf i'})^{-1})F(g)\chi^{-1}(z).$$

\subsection{$p$-adic Automorphic Forms on Unitary Groups}
In this section, we recall the main features of $p$-adic automorphic forms, as discussed in \cite[Chapter 8]{Hida04}.  The reader is encouraged to consult \cite[Chapter 8]{Hida04} for more details.  Let $R$ be a $p$-adic $\ZZ_p$-algebra and let $R_m:=R/p^m$.  Let $S_m = S\times_R R_m$, where $S$ is the ordinary locus of the toroidal compactification of one of the Shimura varieties considered in Section \ref{2.5}.  Following the notation of \cite[Section 8.1.1]{Hida04}, we denote the Igusa variety of level $p^n$ over $S_m$ by $T_{m,n}$, and we define $V_{m, 0}=H^0\left(S_m, \mathcal{O}_m\right)$ and $V_{m,n} = H^0\left(T_{m,n}, \mathcal{O}_{T_{m,n}}\right)$.  So $V_{m, 0}\subseteq V_{m, 1}\subseteq \cdots \subseteq V_{m, n}$.  Also following the notation of \cite[Section 8.1.1]{Hida04}, we put
\begin{align*}
V_{m, \infty} &= \cup_n V_{m, n}\\
V&=\varprojlim_m V_{m, \infty}.
\end{align*}
With $N$ defined to be a unipotent radical of the Borel $B\subseteq \gl_r\left(\ZZ_p\right)\times \gl_s\left(\ZZ_p\right)$ as above, we define the space of $p$-adic modular forms $V_p\left(G, K\right)$ by
\begin{align*}
V_p\left(G, K\right):= V^N.
\end{align*}
(Beware that in the later subsections of \cite[Chapter 8]{Hida04}, the space $V_p(G, K)$ is denoted simply by $V$, even though that space is actually $V^N$, as explained at the beginning of \cite[Section 8.3.2]{Hida04}.)

Before proceeding further, we highlight some facts about rational representations and vector bundles, which we will use in the remainder of the discussion in this section.  Our presentation here follows the conventions of \cite[Section 8.1.2]{Hida04}, which, in turn, refers the reader to \cite[Section I.2]{RAG} and \cite[Section 1.6.5]{GME} for more details.  Given a ring or a sheaf of rings $A$ over a scheme, we define
\begin{align*}
R_A[\uk] = \left\{f: \gl_n/N\rightarrow\mathbb{A}^1| f(ht) = \kappa_{\uk}f(h)\right\},
\end{align*}
where $\kappa_{\uk}$ is the character on the torus corresponding to the weight given by the ordered tuple $\uk$.  As explained on \cite[p. 332]{Hida04}, there is a (unique, up to a $A$-unit multiple) $N$-invariant linear form $\ell_{\can}: R_A[\kappa]\rightarrow A$, which generates $\left(R_A[\kappa]^*\right)^N$.  Note that we can normalize $\ell_{\can}$ so that
\begin{align*}
\ell_{\can}(\phi) = \phi\left(1_n\right)
\end{align*}
for all $\phi\in R_A[\kappa]$.  (Note that $1_n$ denotes the identity matrix in $GL_n/N$.)  When $A$ is a $p$-adic ring, we denote by $\mathcal{C}\left(\gl_n(\ZZ_p)/N(\ZZ_p), A\right)$ the space of $p$-adically continuous $A$-valued functions, and we denote by $\mathcal{LC}\left(\gl_n(\ZZ_p)/N(\ZZ_p), A\right)$ the space of locally constant $A$-valued functions.  Given a ring $A$ of functions on $\gl_n\left(\ZZ_p\right)/N\left(\ZZ_p\right)$, we write $A[\uk]$ to denote the $\kappa_{\uk}$-eigenspace under right multiplication $g\mapsto gt$ by $t\in T(\ZZ_p)$.  Note that for any $p$-adic ring $A$, there is a canonical map
\begin{align*}
R_A[\uk]\hookrightarrow \mathcal{C}\left(\gl_n\left(\ZZ_p\right)/N\left(\ZZ_p\right), A\right)[\uk],
\end{align*}
 and when $A$ is finite $\mathcal{C}\left(\gl_n(\ZZ_p)/N(\ZZ_p), A\right) = \mathcal{LC}\left(\gl_n(\ZZ_p)/N(\ZZ_p), A\right)$.
(Note that the modules $R_A[\uk]$ and $R_A[\uk]^\vee$ are the same as the modules $L_{\uk}$ and $L^{\uk}$, respectively, but in order to make the connection with \cite{Hida04} transparent, we adhere to Hida's notation - rather than the conventions in \cite{Hsieh} -  in this section.)

Let $\uo_m$ denote the pullback of $\uo$ to $S_m$, and let $\uo_{m, \underline{k}}$ denote the pullback of $\uo_{\underline{k}}$ to $S_m$.  As explained on \cite[p. 334]{Hida04}, there is a canonical map
\begin{align*}
 \omega_{\can}^{\uk}: H^0\left(S_m, \uo_{m, \uk}\right)\rightarrow H^0\left(T_{m,m}, R_{T_{m,m}}[\uk]\right).
\end{align*}
  As explained on \cite[p. 335]{Hida04}, there is a natural map
\begin{align*}
\beta: H^0\left(S_m, \uo\right)\rightarrow V_{m, \infty}^N[\underline{k}],
\end{align*}
where $V_{m, \infty}^N[\underline{k}]$ denotes the $\underline{k}$-eigenspace for the action of $T$ on $V_{m, \infty}^N$; and if $m=\infty$, then $\beta$ is injective.  Moreover, letting $R_m' = \oplus_{\underline{k}>>0} H^0\left(S_m, \uo_m\right)$ (where $>>$ means sufficiently regular, in the sense of \cite[Section 5.1.3]{Hida04}), there is a map
\begin{align*}
\beta(m): R_m'\rightarrow V_{m,m}^N,
\end{align*}
for $1\leq m< \infty$, defined by
\begin{align*}
\beta(m)\left(\sum_{\underline{k}>>0}f_{\underline{k}}\right) = \left\{(X/T_{m,m}, j)\mapsto \sum_{\underline{k}}\ell_{\can}\left(\omega_{\can}^{\uk}\left(f_{\uk}\left(X, j\right)\right)\right)\right\}.
\end{align*}
Continuing to follow the notation of \cite[Section 8.1.4]{Hida04}, we note that by taking projective limits, we obtain a map
\begin{align*}
\beta(\infty): R_\infty'\rightarrow V^N = \varprojlim_m V_{m, \infty}^N.
\end{align*}
Let $V_{\infty,m}=\varinjlim_nV_{n,m}$ and $V_{\infty,\infty}=\varprojlim_mV_{\infty, m}$.  When we need to be precise about the signature $(r, s)$ of the unitary group with which we are working, we write $V_{\infty,\infty}^{(r, s)}$.
Define $V_p(G,K):=V_{\infty,\infty}^N$ to be the space of $p$-adic modular forms.
As explained on \cite[p. 336]{Hida04}, the map $\beta(\infty)$, in turn, induces a map
\begin{align*}
\beta = \beta_{\underline{k}}: H^0\left(S, \uo_{\underline{k}}\right)\hookrightarrow \varprojlim_mH^0\left(S_m, \uo_{\underline{k}}\right)\rightarrow V^N[\underline{k}].
\end{align*}
Furthermore, as explained in \cite[Theorem 8.13(1)]{Hida04}, if $M$ denotes $Sh_K^{(p)}\left(GU, X\right)/\ZZ_p$ and $R= \ZZ_p$ or $R =\IQ_p/\ZZ_p$, then there is a canonical inclusion
\begin{align*}
\beta: \oplus_{\uk}H^0\left(M/\ZZ_p, \uo_{\uk}\otimes_{\ZZ_p}R\right)\hookrightarrow V\otimes_{\ZZ_p}R.
\end{align*}
As explained in \cite[Theorem 8.14(1)]{Hida04}, if the signature of $G$ is $(n,n)$ for some integer $n$ and $M$ denotes the toroidal compactification of $Sh_K^{(p)}\left(GU, X\right)/\ZZ_p$, then
there is an inclusion
\begin{align*}
\beta: \oplus_{\uk}H^0\left(M/\ZZ_p, \uo_{\uk}\otimes_{\ZZ_p} R\right)\hookrightarrow V\otimes R
\end{align*}
(for $R = \ZZ_p$ or $R = \IQ_p/\ZZ_p$) such that (in terms of $q$-expansions, the topic of the next section)
\begin{align*}
\beta(f)(q) = \ell_{can}(f)(q).
\end{align*}

\subsection{Fourier-Jacobi Expansions}\label{FJexpns-section}
\subsubsection{Analytic Fourier-Jacobi Expansions}
We are now going to describe the Fourier-Jacobi expansion for vector-valued automorphic forms on $U(r+1, 1)$. This will only be used for studying the constant terms of these automorphic forms.  As in \cite{Hsieh2}, for any holomorphic automorphic form $f$ on $U(r+1,1)$ we can express the analytic Fourier-Jacobi expansion of $f$ at the standard maximal parabolic $P$ by
$$FJ_P(g,f)=a_0(g,f)+\sum a_\beta(y,g,f)q^\beta.$$
  Here each function $a_\beta(-,g,f):\mathbb{C}^r\rightarrow \mathbb{C}$ is a theta function with complex multiplication by $\mathcal{K}$.

There is also an algebraic theory of Fourier-Jacobi expansions, which is discussed in \cite[Section 3]{Hsieh2} and \cite{LAN}. Let $N_H^1\subset H$ be the set of elements of the form $\begin{pmatrix}1&0\\*&1_r\end{pmatrix}\times\{1\}$. Due to the ambiguity of the choice of standard basis for differentials of Mumford objects, the Fourier-Jacobi expansion is only defined up to translation by elements in $N_H^1$. Let $(g,h)\in C(K)\times \mathbf{H}$ be a $p$-adic cusp. Let $g_1\in G_P(\mathbb{A}_f)$ and $[g_1]$ be the class of $g_1$ in $I_{[g]}^n(\mathbb{C})$ and let $\varphi_{[g_1]}$ be the map from the space of $p$-adic modular forms on $U(r,0)$ to its value at $g_1$. The following comparison between analytic and algebraic Fourier-Jacobi expansions is the main result of \cite{lanalgan} and is recalled in \cite[Equation (3.12)]{Hsieh2}:
$$FJ_P(g_1hg,f)_{N_H^1(\mathbb{C})}=\rho_{\underline{k}}(\mathrm{diag}(1,\Omega_\infty,1))\varphi_{[g_1]}(FJ_{[g]}^h(f))_{N_H^1(\mathbb{C})}.$$

\subsubsection{Algebraic Theory for Fourier-Jacobi expansions}\label{algthy-section}

Now we specialize to the group $GU(r+1,1)$ and define the set of cusps and $p$-adic cusps.

\paragraph{Cusps}
Given an open compact subgroup $K$ as above, we define the set of cusp labels by:
$$C(K):=(GL(X_\mathcal{K})\times G_P(\mathbb{A}_f))N_P(\mathbb{A}_f)\backslash G(\mathbb{A}_f)/K.$$
This is a finite set. We write $[g]$ for the class represented by $g\in G(\mathbb{A}_f)$. For each such $g$ whose $p$-component is $1$ we define $K_P^g=G_P(\mathbb{A}_f)\cap gKg^{-1}$ and denote by $S_{[g]}:=S_{G_P}(K_P^g)$ the corresponding Shimura variety for the group $G_P$ with level group $K_P^g$. By strong approximation, we can choose a set $\underline{C}(K)$ of representatives of $C(K)$ consisting of elements $g=pk^0$ for $p\in P(\mathbb{A}_f^{(pN_0)})$ and $k^0\in K^0$ for $K^0$ the maximal compact subgroup.

\paragraph{$p$-adic Cusps}

  As in \cite{Hsieh}, each pair $(g,h)\in C(K)\times\mathbf{H}$ can be regarded as a $p$-adic cusp, i.e. a cusp of the Igusa tower.

  Let $W$ be the Hermitian space for $GU(r,0)$. Recall that $P$ is a maximal parabolic of $G$ with levi $\begin{pmatrix}a&&\\&g&\\&&\bar{a}^{-1}\mu(g)\end{pmatrix}$ and $g\in GU(W)$ and we write $G_P$ for $GU(W)$.

  We can also define Shimura varieties and Igusa towers for the group $G_P$. Recall $K_P^g=G_P(\mathbb{A}_f)\cap gKg^{-1}$. Write $S_{[g]}:=S_{G_P}(K_P^g)$. One defines a group scheme $\mathcal{Z}_{[g]}$ over $S_{[g]}$ using the universal abelian variety as in \cite[Section 2.7]{Hsieh} and denote $\mathcal{Z}_{[g]}^\circ$ the connected component. In our case this is an abelian variety over $S_{G_P}(K_P^g)$. For $\bullet=0,1,\emptyset$ let $K_{P,\bullet}^{g,n}:=gK_\bullet^n g^{-1}\cap G_P(\mathbb{A}_f)$ and $I_{[g]}(K_\bullet^n):=I_{G_P}(K_{P,\bullet}^{g,n})$ be the associated Igusa scheme over $S_{[g]}$. These schemes are affine. Write $A_{[g]}$, $A_{[g]}^n$ for be the coordinate rings for $S_{[g]}$ and $I_{[g]}(K_1^n)$.
In fact we are mainly interested in forms with this level $K_1^n$.  We use $\Lambda$-adic Fourier-Jacobi expansions to study these forms. In particular, the Klingen Eisenstein series we construct are of this level.

  At a cusp $[g]$, one defines a Mumford family $(\mathcal{M}, \iota_\mathcal{M})$; see \cite[Section 2.7]{Hsieh} (or \cite{Eischen} or \cite{Hida04}) for more details. We also fix a $p$-power level structure $j_\mathcal{M}$ as in \cite{Hsieh}.  We define $\mathcal{R}_{[g],\infty}$ to be the ring consisting of formal $q$-expansions $\sum_{\beta\in \mathscr{S}_{[g]}}a_{[g]}^h(\beta)q^\beta$ with
$$a_{[g]}^h(\beta)\in (V_{\infty,\infty}^{(r,0),N})\otimes_{S_{[g]}} H^0(\mathcal{Z}_{[g]}^\circ,\mathcal{L}(\beta)).$$
(This ring is not quite the one defined in \cite[Section 2.7]{Hsieh}.  It is easy to see, however, that this ring contains the rings considered there; and if we want to study $q$-expansions of $p$-adic modular forms of coefficients not necessarily of $p$-power torsion, then we need our ring $\mathcal{R}_{[g],\infty}$.)  One obtains the Fourier-Jacobi expansion for a modular form by evaluating the form at the Mumford family $(\mathcal{M},\omega_\mathcal{M})$ over the base ring $\mathcal{R}_{[g],\infty}$.  (Here, the basis of the differential $\omega_\mathcal{M}$ is not canonical; thus, if our form is not of scalar weight then the Fourier-Jacobi expansion takes values in the orbit of a vector instead of a vector.  For details see \cite[Section 2.7]{Hsieh}.)   Now we recall the Fourier-Jacobi expansion for $p$-adic modular forms following \cite[Section 3.6.2]{Hsieh}. Any $(g,h)\in C(K)\times \mathbf{H}$ can be regarded as a cusp on the Igusa tower. In \emph{loc.cit}, Hsieh defines a subgroup $\mathscr{S}_{[g]}$ of $\mathbb{Q}^+$ (corresponding to $K_1^n$ there).

  For any $f\in V_{\infty,\infty}^{(r+1,1),N}$, the $p$-adic Fourier-Jacobi expansion at the cusp $(g,h)$ is given by
$$\hat{FJ}_{[g]}^h(f):=f(\underline{\mathcal{M}}_{[g]},h^{-1}j_{\mathcal{M}})=\sum_{\beta\in\mathscr{S}_{[g]}}\hat{\mathbf{a}}_{[g]}^h(\beta,f)q^\beta
\in \mathcal{R}_{[g],\infty}.$$
\begin{remark}
The use of ``$\hat{FJ},$'' ``$\hat{M}$'', etc. in the $p$-adic context - while non-standard - follows the convention of Hsieh.  (So note that $\hat{ }$ means to consider the $p$-adic realization of an object or space; it does not denote a Fourier transform.)
\end{remark}

\subsubsection{Siegel Operators}

  We have a Siegel $\Phi$ operator defined by taking the term where $\beta=0$.
The Siegel operator $\Phi_{P, g}$ can be defined analytically as follows. For any $g\in G(\mathbb{A}_\mathbb{Q})$ we define
$$\Phi_{P,g}(f)=\int_{N_P(\mathbb{Q})\backslash N_P(\mathbb{A}_\mathbb{Q})}f\left(n\cdot\left({\bf \tilde{i}},g\right)\right)dn.$$

\subsection{Inner Products of Automorphic Forms on Definite Unitary Groups}\label{2.12}
Recall from Section \ref{geomforms-section} that we have a natural $\mathrm{GL}_r$-equivariant pairing $$\langle,\rangle: L_{\underline{k}}(R)\times L^{\underline{k}}(R)\rightarrow R.$$  There is a natural pairing $M_{\underline{k}}^{(r,0)}(K,R)\times M_{-\underline{k}}^{(0,r)}(K,R)\rightarrow R$ given by
$$\langle f,h\rangle :=\int_{U(\mathbb{Q})\backslash U(\mathbb{A}_\mathbb{Q})}\langle f(g),h(g)\rangle dg.$$
(We identify $U(r,0)$ with $U(0,r)$. Note that due to the different definitions of the automorphy factors for $U(r,0)$ and $U(0,r)$, the infinity types for the two spaces are in fact dual to each other.)\\

  There is a pairing between $p$-adic automorphic forms which we are going to describe below.  First, though, we give another description of $p$-adic automorphic forms on definite unitary groups ($U(r,0)$ or $U(0,r)$). Let $\hat{M}_{\underline{k}}^{(r,0)}(K,R)$ be the space of functions $f:GU(r,0)(\mathbb{A}_f)\rightarrow L_{\underline{k}}(R)$ such that
$$f(\alpha gk)=\rho_{\underline{k}}(k_p^{-1})f(g), \alpha\in G(\mathbb{Q}),k\in K.$$
Then if $n\geq m$ we have $\hat{M}_{\underline{k}}(K_1^n,R_m)=V_{\underline{k}}(K_1^n,R_m)$. If $p$ is invertible in $R$ then we have an isomorphism
$$M_{\underline{k}}^{(r,0)}(K,R)\simeq \hat{M}_{\underline{k}}^{(r,0)}(K,R)$$
$$f\mapsto \hat{f},\hat{f}(g):=\rho_{\underline{k}}(g_p^{-1})f(g).$$
This means the $\hat{M}$ can be viewed of as the image of classical forms of weight $\underline{k}$ in the space of $p$-adic automorphic forms on weight $\underline{k}$. We have similar definitions for the group $GU(0,r)$.
We define the pairing between $p$-adic automorphic forms on $GU(r,0)$ and $GU(0,r)$ by
$$\langle \hat{f},\hat{h}\rangle :=\mathrm{Vol}(K_0)^{-1}\sum_{x\in U(\mathbb{Q})\backslash U(\mathbb{A}_f)/K_0} \langle \hat{f}(x),\hat{h}(x)\rangle.$$
Where $K_0$ is the compact open subgroup of $U(\mathbb{A}_f)$ which is $K_0(p)$ at $p$ and equals the local component for $K$ at all other primes.
One sees that if $f$ and $\hat{f}$, $h$ and $\hat{h}$ correspond to each other in the above sense, then
$$\langle f,h\rangle =\langle \hat{f},\hat{h}\rangle .$$
\begin{remark}
If we identify the group $U(r,0)$ with $U(0,r)$ then an automorphic form for $U(r,0)$ of weight $(a_1,\ldots,a_r)$ is a form of weight $(-a_r,\ldots,-a_1)$ for $U(0,r)$.
\end{remark}
\subsection{Hecke Operators at $p$}\label{hecke-section}
\begin{definition}\label{defkappai}
Let $n=r+s$ and $\underline{k}=(a_1,\cdots,a_r;b_1,\ldots,b_s)$ be a weight. We define the set $\{\kappa_1,\ldots,\kappa_{r+s}\}$ (such that $\kappa_1\geq \cdots\geq \kappa_{r+s}$) to be
$$\left\{b_s+s-1-\frac{n}{2}+\frac{1}{2},b_{s-1}+s-2-\frac{n}{2}+\frac{1}{2},\ldots,b_1-\frac{n}{2}+\frac{1}{2},
-a_r+r+s-1-\frac{n}{2}+\frac{1}{2},\ldots,-a_1+s-\frac{n}{2}+\frac{1}{2}\right\}.$$
\end{definition}

We now briefly introduce some Hecke operators.  (For more details on Hecke operators, see \cite[Section 8.1.6]{Hida04}.)
Let $\mathcal{A}_p:=\mathbb{Z}_p[t_1,t_2,\ldots,t_n,t_n^{-1}]
$ be the Atkin-Lehner ring of $G(\mathbb{Q}_p)$, where $t_i$ is defined by $t_i=N(\mathbb{Z}_p)\alpha_i
N(\mathbb{Z}_p)$, $\alpha_i=\begin{pmatrix}1_{n-i}&\\&p 1_i\end{pmatrix}$.  Writing $N\alpha_i N = \sqcup_u N \alpha_{i, u}$, we have that $t_i$ acts on $\pi^{N(\mathbb{Z}_p)}$ by $$v|t_i=\sum_{u}\alpha_{i, u}^{-1}v.$$
We also define a normalized action with respect to the weight $\underline{k}$ following (\cite[Chapter 8]{Hida04}):
$$v\|t_i:=\delta_B(\alpha_i)^{-1/2}p^{\kappa_1+\ldots+\kappa_i}v|t_i,$$
where $\delta_B$ denotes the modulus character for the Borel $B$.  We denote $||_{t_i}$ as $U_{t_i}$ and define the Hecke operator $U_{t^+}$ to be the composition of the $||_{t_i}$'s above.

\subsection{Differential Operators}
Let $S/T$ be either the Igusa or Shimura variety, and let $A/S$ be the universal Abelian variety.

\subsubsection{The Gauss Manin connection}
Let $\pi:X\rightarrow S$ be a smooth proper morphism of schemes, and let $S$ be a smooth scheme over a scheme $T$. Then the Gauss-Manin connection is a map:
$$\nabla:H_{DR}^q(X/S)\rightarrow H_{DR}^q(X/S).$$
By using the chain rule, we can also define $\nabla:T^\bullet(H_{DR}^{1\pm}(A/S))\rightarrow T^\bullet (H_{DR}^{1\pm}(A/S))\otimes \Omega^1_{S/T}$.  (Here, $T^\bullet$ denotes $T^{\otimes k}$ for some nonnegative integer $k$, and like in \cite{Eischen}, $H_{DR}^{1+}$ (resp. $H_{DR}^{1-}$) denotes the submodule on which $\alpha\mathcal{K}$ acts via multiplication by $\alpha$ (resp. $\bar{\alpha}$).)

\subsubsection{Algebraic Differential Operators}

  As explained in \cite[Section 7]{Eischen}, there is a purely algebraic differential operator
\begin{align*}
D^\rho_{A/S}: H^1_{DR}(A/S)^\rho\otimes T^\bullet\left(H^+\left(A/S\right)\otimes H^-\left(A/S\right)\right)\rightarrow H^1_{DR}\left(A/S\right)^\rho\otimes T^{\bullet+1}\left(H^+\left(A/S\right)\otimes H^-\left(A/S\right)\right),
\end{align*}
which is constructed from the Gauss-Manin connection and the Kodaira-Spencer morphism.  This operator is a generalization to the case of automorphic forms on unitary groups of signature $(n,n)$ of the differential operators constructed for Hilbert modular forms in \cite{Katz78}.

\subsubsection{$C^\infty$ Differential Operators}

 Over $\IC$, there is a canonical splitting
$$H_{DR}^1(C^\infty)=\underline{\omega}(C^\infty)\oplus \mathrm{Split}(C^\infty)$$
of the Hodge filtration corresponding to the holomorphic and anti-holomorphic one-forms. Let $\rho=\rho_-\otimes\rho_+$ be a representation of $GL_n\times GL_n$ which is a quotient of $\rho_{st}^{\otimes d_1}\otimes \rho_{st}^{\otimes d_2}$.  There is a $C^\infty$-differential operator
\begin{align*}
\partial(\rho,C^\infty,d): \left(\uo^-\right)^{\rho_-}\otimes\left(\uo^+\right)^{\rho_+}\rightarrow \left(\uo^-\right)^{\rho_-}\otimes\left(\uo^+\right)^{\rho_+}\otimes\left(\uo^+\otimes\uo^-\right)^{\otimes d}.
\end{align*}
 defined in \cite[Section 8]{Eischen} as the composition
\begin{align*}
&(\underline{\omega}^-)^{\rho_-}\otimes(\underline{\omega}^+)^{\rho_+}\hookrightarrow (H_{DR}^1(\mathcal{A}/S))\rightarrow (H_{DR}^1(\mathcal{A}/S))^{\mathrm{Ind}_P^G\rho}\otimes T^d(H^+(\mathcal{A}/S)\otimes H^-(\mathcal{A}/S))&\\&\rightarrow (\underline{\omega}^-)^{\rho_-}\otimes (\underline{\omega}^+)^{\rho_+}\otimes(\underline{\omega}^+\otimes \underline{\omega}^-)^d,&
\end{align*}
where the first arrow is the canonical inclusion, the second is $(D^\rho)^d$, and the third is $\mod\mathrm{Split}(C^\infty)$.

\subsubsection{$p$-adic Differential Operators}

  Now let $S$ be the Igusa variety over a $p$-adic ring. Since it classifies ordinary abelian varieties with additional structures there is a unit root splitting:
$$H_{DR}^1(\mathcal{A}/S)=\underline{\omega}\oplus\underline{U}$$
where $\underline{U}$ is the unit root subspace. We can define a $p$-adic differential operator $\partial(\rho,p-\mathrm{adic},d)$
$$(\underline{\omega}^-)^{\rho_-}\otimes(\underline{\omega}^+)^{\rho_+}\rightarrow (\underline{\omega}^-)^{\rho_-}\otimes(\underline{\omega}^+)^{\rho_+}\otimes(\underline{\omega}^+\otimes \underline{\omega}^-)^d.$$
similarly to how we defined the $C^\infty$ differential operator, but with $\mathrm{Split}(C^{\infty})$ replaced by $\underline{U}$.  Details are discussed in \cite[Section 9]{Eischen}.

\section{Eisenstein Series and Pullback Formulas}\label{EseriesPullbkFormulas-section}

\subsection{Klingen Eisenstein Series}
Recall that in Equation \eqref{GU-defn}, we defined $GU=GU(r,s)$. Let $\mathfrak{g}\mathfrak{u}(\mathbb{R})$ be the Lie algebra of $GU(r,s)(\mathbb{R})$. We will often use the notation $a,b$ for $a=r-s$ and $b=s$.

\subsubsection{Archimedean Picture}

Now let $(\pi,H)$ be a unitary Hilbert representation of $GU(\mathbb{R})$ with $H_\infty$ the space of smooth vectors. We define a representation of $P(\mathbb{R})$ on $H_\infty$ as follows: for $p=mn,n\in N_P(\mathbb{R}),m=m(g,a)\in M_P(\mathbb{R})$ with $a\in \mathbb{C}^\times$ and $g\in GU(\mathbb{R})$ (with the notation $m(g, a)$ as in Equation \eqref{r1-embedding}) put
\begin{align}
\rho(p)v:=\tau(a)\pi(g)v,v\in H_\infty.\label{rhopi}
\end{align}
(Note that when we work in the global setting below, we will denote $\rho$ from Equation \eqref{rhopi} by $\rho_w$ with $w$ an archimedean place.)  We define a representation by smooth induction $I(H_\infty):=\mathrm{Ind}_{P(\mathbb{R})}^{GU(r+1,s+1)(\mathbb{R})}\rho$ and denote by $I(\rho)$ the space of $K_\infty$-finite vectors in $I(H_\infty)$.\footnote{We shall typically use the notation $I(\rho)$ for induced representations in the context of Klingen Eisenstein series, in contrast to the notation $I_n(\tau)$ that we use for induced representations in the context of Siegel Eisenstein series.}  We also define for each $s\in \mathbb{C}$ a {\it Klingen section}
$$f_s(g):=\delta(m)^{\frac{a+2b+1}{2}+s}\rho(m)f(k),g=mk\in P(\mathbb{R})K_\infty,$$
\index{$\delta(m)$} where $\delta$ is such that $\delta^{a+2b+1}=\delta_P$ for $\delta_P$ \index{$\delta_P$} the modulus character of $P$.  We define an action of $GU(r+1,s+1)(\mathbb{R})$ on $I\left(H_\infty\right)$ by
$$(\sigma(\rho,s)(g)f)(k):=f(kg).$$
  Let $(\pi^\vee,V)$ \index{$\pi^\vee$} be the irreducible $(\mathfrak{gu}(\mathbb{R}),K_\infty')$-module given by $\pi^\vee(x)=\pi(\eta^{-1}x\eta)$ for $\eta=\begin{pmatrix}&&1_b\\&1_a&\\-1_b&&\end{pmatrix}$ and
$x$ in $\mathfrak{gu}(\mathbb{R})$ or $K_\infty'$, and denote $\rho^\vee,I(\rho^\vee), I^\vee(H_\infty)$ \index{$\rho^\vee$} and $\sigma(\rho^\vee,s),I(\rho^\vee))$ the representations and spaces defined as above but with $\pi$ and $\tau$ replaced by $\pi^\vee\otimes(\tau\circ \det )$ and $\bar{\tau}^c$, respectively.  We are going to define an intertwining operator. Let
\begin{align}\label{wdefn}
w=\begin{pmatrix}&&1_{b+1}\\&1_a&\\-1_{b+1}&&\end{pmatrix}.
\end{align}
For any $s\in\mathbb{C}$, $f\in I(H_\infty)$, and $k\in K_\infty$, consider the integral
\begin{equation}
A(\rho,s,f)(k):=\int_{N_P(\mathbb{R})}f_s(wnk)dn.
\end{equation}
This is absolutely convergent when $Re(s)>\frac{a+2b+1}{2}$ and $A(\rho,s,-)\in \Hom_{\IC}(I(H_\infty),I^\vee(H_\infty))$ intertwines the actions of $\sigma(\rho,s)$ and $\sigma(\rho^\vee,-s)$.

\subsubsection{Non-Archimedean Picture}

  Our discussion here follows \cite[9.1.2]{SU}. Let $(\pi,V)$ be an irreducible, admissible representation of
$GU(\mathbb{Q}_v)$ which is unitary and tempered. Let $\psi$ and $\tau$ be
unitary characters of $\mathcal{K}_v^\times$ such that
$\psi$ is the central character for $\pi$. We define a representation $\rho$ of $P(\mathbb{Q}_v)$ as follows. For $p=mn,n\in N_P(\mathbb{Q}_v)$,
$m=m(g,a)\in M_P(F_v), a\in K_v^\times,g\in GU(\mathbb{Q}_v)$ let
\begin{align}
\rho(p)v:=\tau(a)\pi(g)v,v\in V.\label{rhopi-nonarch}
\end{align}
(Note that when we work in the global setting below, we will denote $\rho$ from Equation \eqref{rhopi-nonarch} by $\rho_w$ with $w$ a non-archimedean place.)
Let $I(\rho)$ be the representation defined by admissible induction:
$I(\rho)=\mathrm{Ind}_{P(\mathbb{Q}_v)}^{GU(r+1,s+1)(\mathbb{Q}_v)}\rho$.
As in the Archimedean case, for each $f\in I(\rho)$ and each
$s\in \mathbb{C}$ we define a function $f_s$ on $GU(\mathbb{Q}_v)$ by
$$f_s(g):=\delta(m)^{\frac{a+2b+1}{2}+s}\rho(m)f(k),g=mk\in P(\mathbb{Q}_v)K_v$$
and a representation $\sigma(\rho,s)$ of $GU(r+1,s+1)(F_v)$ on $I(\rho)$
by
$$(\sigma(\rho,s)(g)f)(k):=f_s(kg).$$

  Let $(\pi^\vee, V)$ be given by $\pi^\vee(g)=\pi(\eta^{-1}g\eta).$
This representation is also tempered and unitary. We denote by
$\rho^\vee,I(\rho^\vee)$, and $(\sigma(\rho^\vee,s),I(\rho^\vee))$
the representations and spaces defined as above but with $\pi$
and $\tau$ replaced by $\pi^\vee\otimes(\tau\circ
\det)$, and $\bar \tau^c$, respectively.

  For $f\in I(\rho),k\in
K_v$, and $s\in \mathbb{C}$ consider the integral
\begin{equation}
A(\rho,s,f)(k):=\int_{N_P(F_v)}f_s(wnk)dn.
\end{equation}
As a consequence of our hypotheses on $\pi$, this integral converges
absolutely and uniformly for $s$ and $k$ in compact subsets of
$\{s:Re(s)>(a+2b+1)/2\}\times K_v$. Moreover, for such $s$,
$A(\rho,s,f)\in I(\rho^\vee)$ and the operator $A(\rho,s,-)\in
\Hom_\mathbb{C}(I(\rho),I(\rho^\vee))$ intertwines the actions of
$\sigma(\rho,s)$ and $\sigma(\rho^\vee,-s).$

For each open subgroup $U\subseteq K_v$, let $I(\rho)^U\subseteq
I(\rho)$ be the finite-dimensional subspace consisting of functions
satisfying $f(ku)=f(k)$ for all $u\in U$. Then the function
\begin{align*}
\left\{s\in
\mathbb{C}:Re(s)>(a+2b+1)/2\right\}&\rightarrow \Hom_\mathbb{C}(I(\rho)^U,
I(\rho^\vee)^U)\\ s&\mapsto A(\rho,s,-)
\end{align*}
 is holomorphic. This map
has a meromorphic continuation to all of $\mathbb{C}$.

  Note that when $\pi$ and $\tau$ are unramified, there is a unique (up to scalar) unramified vector $F_v^{\sph}\in I(\rho)$. \index{$F_{\rho_v}$}

\subsubsection{Global Picture}

  We follow \cite[9.1.4]{SU} for this part. Let $(\pi,V)$ be an irreducible cuspidal tempered automorphic representation of $GU(\mathbb{A}_\mathbb{Q})$. It is an admissible $(\mathfrak{gu}(\mathbb{R}),K_\infty')\times GU(\mathbb{A}_f)$-module which is a restricted tensor product of local irreducible admissible representations. Let $\psi, \tau:\mathbb{A}_\mathcal{K}^\times \rightarrow \mathbb{C}^\times$ be Hecke characters such that $\psi$ is the central character of $\pi$. Let $\tau=\otimes \tau_w$ and $\psi=\otimes \psi_w$ be their local decompositions, $w$ running over places of $\mathbb{Q}$. Define a representation $\rho$ of $(P(\mathbb{Q}_\infty)\cap K_\infty)\times P(\mathbb{A}_{\mathbb{Q},f})$ by
 $$\rho(p)v:=\otimes(\rho_w(p_w)v_w),$$
where $\rho_w$ denotes the representation at the place $w$ that was denoted above simply by $\rho$ and, similarly, $v_w$ denotes a vector that was denoted above simply by $v$.  Let $I(\rho)$ be the restricted product $\otimes' I(\rho_w)$ with respect to those $w$ at which
 $\tau_w,\psi_w,\pi_w$ are unramified.  For each $s\in \mathbb{C}$ and $f\in I(\rho)$, we define a function $f_s$ on $GU(r+1,s+1)(\mathbb{A})$ by
 $$f_s(g):=\otimes f_{w,s}(g_w)$$
 where each section $f_{w,s}$ is the local section at $w$ denoted simply by $f_s$ earlier, and we define an action $\sigma(\rho,s)$ of $(\mathfrak{gu},K_{\infty})\otimes GU(r+1,s+1)(\mathbb{A}_f)$ on $I(\rho)$ by $\sigma(\rho,s):=\otimes\sigma(\rho_w,s).$  We write $I(\rho, s)$ for the space of all such $f_s$.
 Similarly, we define $\rho^\vee,I(\rho^\vee)$, and $\sigma(\rho^\vee,s)$ the representations and spaces defined as above but with $\pi$ and $\tau$ replaced by $\pi^\vee\otimes(\tau\circ \det )$ and $\bar{\tau}^c$, respectively.

\subsubsection{Klingen Eisenstein Series}

  Let $\pi,\psi,$ and $\tau$ be as in the above subsection. For $f\in I(\rho),s\in
 \mathbb{C}$, there are maps from
$I(\rho)$ and $I(\rho^\vee)$ to spaces of automorphic forms on $P(\mathbb{A}_\IQ)$ given by
$$f\mapsto (g\mapsto f_s(g)(1)).$$
In the following we often write $f_s$ for the automorphic form given by this recipe.

If $g\in GU(r+1,s+1)(\mathbb{A}_\IQ)$ it is well known that
\begin{equation}
E(f,s,g):=\sum_{\gamma\in P(\IQ)\setminus G(\IQ)} f_s(\gamma g)
\end{equation}
converges absolutely and uniformly for $(s,g)$ in compact subsets of $\{s\in\mathbb{C}:Re(s)>\frac{a+2b+1}{2}\}\times
GU(r+1,s+1)(\mathbb{A})$. The automorphic forms $E(f, s, g)$ are called {\it Klingen Eisenstein series}.

\begin{definition}\index{$E_P$} \index{$E_R$}
For any parabolic subgroup $R$ of $GU(r+1,s+1)$ and any automorphic form $\varphi$, we denote by $\varphi_R$ the constant term of $\varphi$ along $R$.
\end{definition}
The following lemma is well-known.  (See \cite[lemma 9.1.6]{SU}.)
\begin{lemma}\label{constant}
Let $R$ be a standard $F$-parabolic of $GU(r+1,s+1)$ (i.e. $R\supseteq B$ where $B$ is the standard Borel). Suppose $Re(s)>\frac{a+2b+1}{2}$.
\begin{enumerate}
\item[(i)]{If $R\not= P$ then $E(f,s,g)_R=0$;}
\item[(ii)]{$E(f,s,-)_P=f_s+A(\rho,f,s)_{-s}$.}
\end{enumerate}
\end{lemma}

\subsection{Siegel Eisenstein Series on $GU_n$}
\subsubsection{Local Picture}

  Our discussion in this section follows \cite[Sections 11.1-11.3]{SU} closely.  For each positive integer $n$, let $Q=Q_n$ \index{$Q_n$} be the Siegel parabolic subgroup of $GU_n=GU(n,n)$ consisting of matrices
$\begin{pmatrix}A&B\\0\cdot 1_n&D\end{pmatrix}$.
We denote by $K_{n,v}$ the maximal compact subgroup $GU_n\left(\ZZ_v\right)\subseteq GU_n\left(\IQ_v\right)$.  For each positive integer $n$, place $v$ of $\mathbb{Q}$, and character $\tau$ of $\mathcal{K}_v^\times$, we let $I_n(\tau)$ \index{$I_n(\tau)$} denote the space of smooth
$K_{n,v}$-finite functions\index{$K_{n,v}$}
\begin{align*}
f: K_{n,v}\rightarrow \mathbb{C}
\end{align*}
such that $f(qk)=\tau(\det D_q)f(k)$ for all $q\in Q_n(\mathbb{Q}_v)\cap K_{n,v}$.  (We write $q$ as block matrix $q=\begin{pmatrix}A_q&B_q\\0&D_q\end{pmatrix}$.) For $s\in \mathbb{C}$ and $f\in I_n(\tau)$, we also define a function
\begin{align*}
f(s,-):GU_n(\mathbb{Q}_v)\rightarrow \mathbb{C}
\end{align*}
 by
 \begin{align*}
 f(s,qk):=\tau(\det D_q))\left|\det A_q D_q^{-1}\right|_v^{s+n/2}f(k),
 \end{align*}
for all $q\in Q_n(\mathbb{Q}_v)$  and $k\in K_{n,v}.$

  Let $\tau$ be a unitary character of $\mathcal{K}_v^\times$, and let $v$ be a place of $\mathbb{Q}$. For $f\in I_n(\tau),s\in \mathbb{C},$ and $k\in K_{n,v}$, the intertwining integral is defined by
$$M(s,f)(k):=\bar\tau^n(\mu(k))\int_{N_{Q_n}(F_v)}f(s,w_nrk)dr,$$
where $N_{Q_n} = N_Q$ denotes the unipotent radical of the parabolic subgroup $Q_n = Q$ and $w_n=\begin{pmatrix}0 & 1_n\\ -1_n & 0\end{pmatrix}$.  For $s$ in compact subsets of $\{Re(s)>n/2\}$, this integral converges absolutely and uniformly, with the convergence being uniform in $k$. In this case it is easy to see that $M(s,f)\in I_n(\bar{\tau}^c).$  A standard fact from the theory of Eisenstein series says that this has a continuation to a meromorphic section on all of $\mathbb{C}$.

  Let $\mathcal{U}\subseteq \mathbb{C}$ be an open set. By a meromorphic section of $I_n(\tau)$ on $\mathcal{U}$, we mean a function $\varphi:\mathcal{U}\mapsto I_n(\tau)$ taking values in a finite-dimensional subspace $V\subset I_n(\tau)$ and such that
$\varphi:\mathcal{U}\rightarrow V$ is meromorphic.

\subsubsection{Global Picture}
For an id\`ele class character $\tau=\otimes\tau_v$ of $\mathbb{A}_\mathcal{K}^\times$ we define the space $I_n(\tau)$ to be the restricted tensor product defined using the spherical vectors $f_v^{\sph}\in I_n(\tau_v),f_v^{\sph}(K_{n,v})=1$ at the finite places $v$ where $\chi_v$ is unramified.

  For $f\in I_n(\tau)$ we consider the Siegel Eisenstein series
\begin{equation}\label{SiegelEisenstein}
E(f,s,g):=\sum_{\gamma\in Q_n(\mathbb{Q})\setminus GU_n(\mathbb{Q})} f(s,\gamma g).
\end{equation}
This series converges absolutely and uniformly for $(s,g)$ in compact subsets of $\{Re(s)>n/2\}\times GU_n(\mathbb{A}_\mathbb{Q})$.

  The Eisenstein series $E(f,s,g)$ has a meromorphic continuation in $s$ to all of $\mathbb{C}$ in the following sense. If $\varphi:\mathcal{U}\rightarrow I_n(\tau)$ is a meromorphic section, then we put
\begin{align*}
E(\varphi, s,g)=E(\varphi(s), s,g).
\end{align*}
This is defined at least on the region of absolute convergence and it is well known that it can be meromorphically continued to all $s\in \mathbb{C}$.

  Now for $f\in I_n(\tau),s\in \mathbb{C}$ and $k\in \prod_{v\nmid \infty}K_{n,v}\prod_{v|\infty}K_\infty$, there is a similar intertwining integral $M(s,f)(k)$ as above but with the integration being over $N_{Q_n}(\mathbb{A}_\mathbb{Q})$. This again converges absolutely and uniformly for $s$ in compact subsets of $\{Re(s)>n/2\}\times K_n$. Thus $s\mapsto M(s,f)$ defines a holomorphic section $\{Re(s)>n/2\}\rightarrow \Hom(I_n(\tau)\rightarrow I_n(\bar{\tau}^c))$. This has a continuation to a meromorphic section on $\mathbb{C}$. For $Re(s)>n/2$, we have
$$M(s,f)=\otimes_v M(s,f_v),f=\otimes f_v.$$
\index{$M(z;-)$}

  The functional equation for the Siegel Eisenstein series is
$$E(f,s, g)=\tau^n(\mu(g))E(M(s,f);-s,g)$$
in the sense that both sides can be meromorphically continued to all of $\IC$.

\subsubsection{Fourier Coefficients}
\begin{lemma}
Let $f=\otimes_vf_v\in I_n(\tau)$ be such that for some prime $\ell$ the support of $f_\ell$ is in $Q_n(\mathbb{Q}_\ell)w_nQ_n(\mathbb{Q}_\ell)$. Let $\beta$ be an element of the space $S_n\left(\CK_v\right)$ of Hermitian matrices with entries in $\CK_v$, and let $q = \left(q_v\right)$ be an element of $Q_n(\mathbb{A}_\mathbb{Q})$. If $\mathrm{Re}(s)>n/2$, then the Fourier coefficient at $\beta$ is
$$E_\beta(f;s,q)=\prod_v\int_{S_n\left(\CK_v\right)}f_v\left(s,w_nr(S_v)q_v\right)e_v(-\mathrm{Tr}\beta S_v)dS_v,$$
where $r(S_v)=\begin{pmatrix}1_n & S_v\\ 0 & 1_n\end{pmatrix}$.  The right hand side is absolutely convergent for $\mathrm{Re}(s)>n/2$.
\end{lemma}
Recall the definitions of $S$, $S'$, and the embedding $\alpha$ from Section \ref{unitarygroups-section}. Recall also that $E$ is the Siegel Eisenstein series defined in \eqref{SiegelEisenstein}.

We define $f_{v,\beta}$ to be the local integral on the right hand side of the above expression.
\subsection{Pull-Back Formulas}\label{pbsection}
Let $\tau$ be a unitary idele class character of $\mathbb{A}_\mathcal{K}^\times$. Given a cusp form $\varphi$ on $GU(r,0)$ we consider
\begin{align}\label{Fvarphi-equation}
F_\varphi(f;s,g):=\int_{U(0,r)(\mathbb{A}_\mathbb{Q})} f(s,S^{-1}\mathrm{diag}(g,g_1h)S)\bar\tau(\det g_1h)\varphi(g_1h)dg_1,
\end{align}
where $S$ is defined as in Equation \eqref{Sdefn-equ}, $f\in I_{r+1}(\tau)$, $g\in GU(r+1,1)(\mathbb{A}_\IQ)$, $h\in GU(0, r)(\mathbb{A}_\IQ)$, and $\mu(g)=\mu(h)$.
This is independent of $h$.  We also define
$$F_\varphi'(f';s,g):=\int_{U(0,r)(\mathbb{A}_\mathbb{Q})} f'(s,S^{'-1}\mathrm{diag}(g,g_1h)S')\bar\tau(\det g_1h)\varphi(g_1h)dg_1,$$
where $S'$ is defined as in Equation \eqref{Sprimedefn-equ}, $f'\in I_{r}(\tau)$, $g\in GU(r,0)(\mathbb{A}_\mathbb{Q})$, $h\in GU(r,0)(\mathbb{A}_\IQ)$, and $\mu(g)=\mu(h).$

We will work with the pull-back formulas given in Proposition \ref{pullbackidentities-prop}.
\begin{proposition}\label{pullbackidentities-prop}
Let $\tau$ be a unitary idele class character of $\mathbb{A}_\mathcal{K}^\times$.
\begin{enumerate}
\item[(i)]{ If $f'\in I_{r}(\tau),$ then $F_\varphi'(f';s,g)$ converges absolutely and uniformly for $(s,g)$ in compact sets of $\{Re(s)>r\}\times GU(r,0)(\mathbb{A}_\mathbb{Q})$, and for any $h\in GU(r,s)(\mathbb{A}_\IQ)$ such that $\mu(h)=\mu(g)$
\begin{align}\label{pullbak1-equ}
\int_{U(0,r)(\IQ)\backslash U(0, r)(\mathbb{A}_{\IQ})} E(f';s,S'^{-1}(g,g_1h)S')\bar{\tau}(\det g_1h)\varphi(g_1h)dg_1=F_\varphi'(f';s,g).
\end{align}}
\item[(ii)]{If $f\in I_{r+1}(\tau)$, then $F_\varphi(f;s,g)$ converges absolutely and uniformly for $(s,g)$ in compact sets of $\{Re(s)>r+1/2\}\times GU(r+1,s+1)(\mathbb{A}_\mathbb{Q})$ such that $\mu(h)=\mu(g)$
\begin{align}\label{pullbak2-equ}
\begin{split}
\int_{U(0,r)(\mathbb{Q})\setminus U(0, r)(\mathbb{A}_\mathbb{Q})} E(f;s,S^{-1}(g,g_1h)S)\bar{\tau}&(\det g_1h)\varphi(g_1h)dg_1\\
&=\sum_{\gamma\in P(\mathbb{Q})\setminus G(r+1,1)(\mathbb{Q})} F_\varphi(f;s,\gamma g),
\end{split}
\end{align}
with the series converging absolutely and uniformly for $(s,g)$ in compact subsets of $\{Re(s)>r+1/2\}\times G(r+1,1)(\mathbb{A}_\mathbb{Q}).$}
\end{enumerate}
\end{proposition}
\begin{proof}
This is \cite[Proposition 11.2.3]{SU}; as explained in the proof of \cite[Proposition 11.2.3]{SU}, statement (i) is due to \cite{GPSR},  and statement (ii) is a straight-forward generalization due to Shimura and Garrett.
\end{proof}

\section{Local Computations}\label{local-computations-section}
For the remainder of the paper, we specialize to the case $s=0$.  There are two reasons for our restriction on the signature.  First, the pairings we consider can be expressed as finite sums in this case, which simplifies the construction.  Second, in this case, the Klingen Eisenstein series we construct is holomorphic by Proposition \ref{questionprop}; to make the construction work for arbitrary signature, one needs a holomorphic projection operator (which we have not touched for this project) since the forms obtained by applying a differential operator are merely nearly holomorphic.

Let $d=2(a_1+\cdots+a_r)$, where $a_1\geq \cdots \geq a_r$ as above.
For each positive integer $\kappa\geq r+1$, we will construct a differential operator on the space of holomorphic forms of scalar weight $\kappa$ on $GU(r+1,r+1)$. We consider the representation of $GL_{r+1}\times GL_1\times GL_r$ of highest weight $(a_1,\ldots,a_r,0),(0),(a_1,\ldots,a_r)$. Note that this representation is a summand of $(\mathrm{St}_{\GL_{r+1}}\otimes \mathrm{St}_{\GL_{r+1}})^d$ via the inclusion $\GL_{r+1}\times \GL_1\times \GL_r\hookrightarrow GL_{r+1}\times GL_{r+1}$.
We obtain a differential operator $D_{(\underline{k},0,\kappa)}$ from $\mathcal{M}_\kappa^{(r+1,r+1)}$ to nearly holomorphic forms or $p$-adic modular forms on $U(r+1,1)\times U(0,r)$ of weight $(a_1,\cdot, a_r,0;\kappa), (a_1,\ldots,a_r)$, by applying the following steps:
First, we apply the differential operators $\partial (\rho_\kappa,p-\mathrm{adic},d)$ or $\partial(\rho_\kappa,C^{\infty},d)$ (where $\rho_\kappa$ denotes the representation of weight $\kappa$), and then we pull back to $U(r+1,1)\times U(0,r)$, and finally, we project to the summand of the above representation. \\

  We are also going to construct a differential operator on the space of holomorphic forms on $U(r,r)$ of scalar weight $\kappa$. We consider the representation of $GL_{r}\times GL_r$ with highest weight $(a_1,\ldots,a_r),(a_1,\ldots,a_r)$. Note that this representation is a summand of $(\mathrm{St}_{GL_{r}}\otimes \mathrm{St}_{GL_{r}})^d$.
Similarly to the procedure discussed in the previous paragraph,we obtain a differential operator $D_{(\underline{k},0,\kappa)}$ from $\mathcal{M}_\kappa^{(r,r)}$ to nearly holomorphic forms or $p$-adic modular forms on $U(r, 0)\times U(0,r)$ of weight $(a_1,\ldots,a_r), (a_1,\ldots,a_r)$, by applying the following steps:
First, we apply the differential operator $\partial (\rho_\kappa,p-\mathrm{adic},d)$ or $\partial(\rho_\kappa,C^{\infty},d)$, then we pull back to $U(r,0)\times U(0,r)$, and finally, we project to the summand of the above representation. \\

\subsection{Finite Primes, Unramified Case}
In this subsection we define $f_{v,\sieg}=f_v^{sph}\in I_{r+1}(\tau)$ and $f_{v,\sieg}'=f_v^{sph}\in I_r(\tau)$.
\subsubsection{Fourier Coefficients}
  Let $v\neq p$ be a prime of $\IQ$, and let $\tau$ be a character of $\mathcal{K}_v^\times$, where $\mathcal{K}_v:=\mathcal{K}\otimes \ZZ_v$. For $n=r$ or $r+1$, $f\in I_n(\tau)$, and $\beta\in S_m(F_v),0\leq m\leq n$, the local Fourier-Jacobi coefficient is defined by
$$f_{v,\beta}(s,g):=\int_{S_m(F_v)}f\left(s,w_n\begin{pmatrix}1_n&S\\&1_n\end{pmatrix}g\right)e_v(-\mathrm{Tr}\beta S)dS,$$
as in \cite[Sections 18.9 and 18.10]{Shi97}.

We denote by $\chi_{\mathcal{K}}$ the quadratic character associated to the extension $\mathcal{K}/\IQ$ by class field theory.  We record \cite[Proposition 18.14 and 19.2]{Shi97}, a special case of which appears in \cite[Lemma 11.4.6]{SU}.
\begin{lemma}
Let $\beta\in S_n(F_v)$ and let $r:=rank(\beta)$.  We denote by $\tau'$ the restriction of the character $\tau$ on $\mathcal{K}_p$ to $\IQ_p^\times$.  Then for each $y\in GL_n(\mathcal{K}_v)$,
\begin{eqnarray*}
f_{v,\beta}^{\sph}(s,\mathrm{diag}(y,{}^t\!\bar y^{-1}))=&\tau(\det y)|\det y\bar y|_v^{-s+n/2}\\
&\times \frac{\prod_{i=r}^{n-1} L(2s+i-n+1,\bar\tau'\chi_\mathcal{K}^i)}{\prod_{i=0}^{n-1}L(2s+n-i,\bar\tau\chi_\mathcal{K}^i)}h_{v,{}^t\!\bar y\beta y}(\bar\tau'(\varpi)q_v^{-2s-n}),\end{eqnarray*}
where $\tilde{\pi}$ denotes the contragredient of $\pi$, $\tilde{\tau}$ denotes the contragredient of $\tau$, $\varpi$ is a uniformizer, $q_v$ is the size of the residue field, and $h_{v,{}^t\!\bar y\beta y}\in \mathbb{Z}[X]$ is a monic polynomial depending on $v$ and ${}^t\!\bar y\beta y$ but not on $\tau$. If $\beta\in S_n(\mathcal{O}_v)$ and $\det \beta\in \ZZ_v^\times$, then $h_{v,\beta}=1$; in this case, we say that $\beta$ is \emph{$v$-primitive}.\end{lemma}

\subsubsection{Pullback Integrals}\label{pbints}

Recall the relationship between $\rho$ and $\pi$ from Equation \eqref{rhopi-nonarch}.
\begin{lemma}\label{lemma4.2}
Suppose that $\pi,\psi$, and $\tau$ are unramified and that $\varphi\in\pi$ is a new vector. If $Re(s)>r/2$, then the pull back integral converges, and
$$f_{v,\Kling}(s, g):=F_\varphi(f_v^{\sph}; s, g)=
\frac{L(\tilde\pi,\bar\tau^c,s+1)}{\prod_{i=0}^{r-1}L(2s+r+1-i,\bar\tau'\chi_\mathcal{K}^i)}
F_{\rho,s}(g),$$
and
$$F_\varphi'(f_v^{\sph,'};s,g)=
\frac{L(\tilde\pi,\bar\tau^c,s+\frac{1}{2})}{\prod_{i=0}^{r-1}L(2s+r-i,\bar\tau'\chi_\mathcal{K}^i)}
\varphi(g),$$
where $\varphi\in\pi_v$ is the spherical vector and $F_{\rho, s}$ is the spherical section and recall that $\tilde{\pi}$ is the contragradient representation of $\pi$.
\end{lemma}
A more general version of Lemma \ref{lemma4.2} appears in \cite[Lemma 4.2.3]{WAN}, which cites a proof given in \cite{LR05}.

\subsection{$\ell$-adic Computations}
\subsubsection{Pullback Integrals}

  Let $\ell$ be a rational prime that is prime to $p$.  By \cite[Lemma 4.3.2]{WAN}, there is an element $y$ in $\mathcal{O}_\ell$ (where $\mathcal{O}_\ell:=\mathcal{O}_{\mathcal{K}}\otimes\ZZ_\ell$) divisible by sufficiently high powers of $\ell$ so that Lemma \ref{lemma4.4} and Lemma \ref{lemma4.5} below hold; throughout this section, we take $y$ to be such an element.
\begin{definition}
Let $f^\dag$ be the Siegel section in $I_{r+1}(\tau)$ (resp. $I_r(\tau)$) supported in $Q(\mathbb{Q}_\ell)w_{r+1}N_Q(\mathbb{Z}_\ell)$ (resp. $Q(\mathbb{Q}_\ell)w_{r+1}N_Q(\mathbb{Z}_\ell)$) which takes the value $1$ on $w_{r+1}$ (resp. $w_r$) and is invariant under $N_Q(\mathbb{Z}_\ell)$. We define a Siegel section by
$$f_{\ell,\sieg}(g)=f^\dag(g\tilde{\gamma}_\ell)$$
for $\tilde{\gamma}_\ell=\begin{pmatrix}1&&&\\&1_r&&\frac{1}{y\bar{y}}\\&&1&\\&&&1_r\end{pmatrix}$.
Similarly, we define $f_{\ell,\sieg}'(g)=f^\dag(g\tilde{\gamma}_\ell')$ for $\tilde{\gamma}_\ell'=\begin{pmatrix}1&\frac{1}{y\bar{y}}\\&1\end{pmatrix}$.
\end{definition}
The following two lemmas are proved in \cite[Lemmas 4.3.2 and lemma 4.3.3]{WAN}:
\begin{lemma}\label{lemma4.4}
Let $K^{(2)}_\ell$ be the subgroup of $GU(r+1,1)(\mathbb{Q}_\ell)$ consisting of matrices of the form $\begin{pmatrix}1&f&c\\&1&g\\&&1\end{pmatrix}$ such that $f$ is a $1\times r$ matrix, $g$ is a $r\times 1$ matrix,
$c-\frac{1}{2} f\zeta{ }^t\bar{f}\in\ZZ_\ell$, $g=\zeta{ }^t\bar{f}$, and $f\in \left(y\bar{y}\right)M_{1\times r}\left(\mathcal{O}_\ell\right)$.
Let $\varphi\in\pi_v$. Then the pullback section $F_{\varphi}(s;w,f_{v,\sieg})$ is supported in $PwK_\ell^{(2)}$ and is invariant under the action of $K_\ell^{(2)}$, where $w$ is defined as in Equation \eqref{wdefn}.
\end{lemma}
Now we let $\mathfrak{Y}$ be the subset of $U(r)(\mathbb{Q}_\ell)$ consisting of matrices $g$ such that $(1-g)\in\zeta y\bar{y}(1+y\bar{y}\tilde{N})$ for some $\tilde{N}$ with $\ell$-integral coefficients.
\begin{lemma}\label{lemma4.5}
Let $\varphi\in\pi_\ell$ be a vector invariant under the action of $\mathfrak{Y}$.  Then $F_{\varphi}(s;w,f_{v,\sieg})=\tau(y\bar{y})|(y\bar{y})^2|_\ell^{-s-\frac{r+1}{2}}\mathrm{Vol}(\mathfrak{Y})\varphi$ and $F_\varphi'(f_{v,\sieg}'; s,1)=\tau(y\bar{y})|(y\bar{y})^2|_\ell^{-s-\frac{r}{2}}\mathrm{Vol}(\mathfrak{Y})\phi$.
\end{lemma}
We fix once for all such a vector $\varphi$ and define $f_{v,Kling}=F_{\varphi}(s;w,f_{v,\sieg})$
\subsubsection{Fourier Coefficients}
In Lemma \ref{lemma4.6} below, we let $e_\ell$ denote the exponential function at $\ell$. Let $S_n(\mathbb{Z}_\ell)^*$ be the dual of $S_n(\mathbb{Z}_\ell)$ under the pairing $<g,h>=\mathrm{tr}(gh)$.
\begin{lemma}\label{lemma4.6}
Let $f_v$ be a Siegel section.
\begin{itemize}
\item[(i)] Let $\beta  = \left(\beta_{ij}\right)\in S_{r+1}(\mathbb{Q}_\ell)$. Then $f_{v,\beta}(s,1)=0$ if $\beta\not\in S_{r+1}(\mathbb{Z}_\ell)^*$. If $\beta\in S_{r+1}(\mathbb{Z}_\ell)^*$ then
$$f_{v,\beta}(s,\mathrm{diag}(A,{}^t\!\bar{A}^{-1}))=D_\ell^{-\frac{(r+1)r}{4}}\tau(\det A)|\det A\bar{A}|_\ell^{-s+\frac{r+1}{2}}e_\ell\left(\frac{\beta_{22}+\cdots+\beta_{r+1r+1}}{y\bar{y}}\right)$$
where $D_\ell$ is the discriminant of $\mathcal{K}_\ell/\mathbb{Q}_\ell$.
\item[(ii)] If $\beta=\left(\beta_{ij}\right)\in S_r(\mathbb{Q}_\ell)$, then $f_{v,\beta}(s,1)=0$ if $\beta\in S_r(\mathbb{Z}_\ell)^*$. If $\beta\in S_r(\mathbb{Z}_\ell)^*$ then
$$f_{v,\beta}'(s,\mathrm{diag}(A,{}^t\!\bar{A}^{-1}))=D_\ell^{-\frac{r(r-1)}{4}}\tau(\det A)|\det A\bar{A}|_\ell^{-s+\frac{r}{2}}e_\ell\left(\frac{\beta_{11}+\cdots+\beta_{rr}}{y\bar{y}}\right).$$
\end{itemize}
\end{lemma}
The proof of \ref{lemma4.6} appears in \cite[Lemma 4.3.5]{WAN}.

\subsection{$p$-adic Computations}
Let $\tau_p=(\tau_1,\tau_2)$ be a character of $\mathcal{K}_p^\times = \left(\mathcal{K}\otimes\ZZ_p\right)^\times$, where we have identified $\mathcal{K}_p^\times$ with $\IQ_p^\times\times\IQ_p^\times$ (so $\tau_1$ is a character on the first factor and $\tau_2$ is a character on the second factor).  Suppose $(\tau_1\tau_2)$ has conductor $p$. Then we denote by $\tilde{f}_1\in I_n(\tau)$ the Siegel section supported in $Q(\mathbb{Q}_p)N(\mathbb{Z}_p)$ that takes the value $1$ on the identity and is invariant under right action of $N(\mathbb{Z}_p)$.

\subsubsection{Stabilizations}
  Let $(\chi_1,\ldots,\chi_n)$ be an $r$-tuple of characters of $\mathbb{Q}_p^\times$ whose conductors divide $(p)$. Suppose $(\chi_1,\ldots,\chi_n)$ is regular in the sense of Casselman \cite{casselman}.  (This is a condition to guarantee that the induced representation is irreducible and that by changing the order of the characters we still get the same representation.)  Let $\pi=I(\chi_1,\ldots,\chi_n)$ be the corresponding principal series representation. We consider the space $\pi^{\Gamma_1(p)}$ of vectors invariant under $\Gamma_1(p)$. By the Bruhat decomposition, this is a space of dimension $n!$. By a {\it stabilization} of $\pi$ we mean a vector $v\in\pi$ which is an eigenvector for all the Hecke operators $U_{t_i}$ (depending on a weight $\underline{k}$. This is equivalent to choosing an ordering of the $\chi_i$'s.
(Recall from Definition \ref{defkappai}, that the $\kappa_i$'s are defined in terms of the $\underline{k}$.)
\begin{proposition}
Let $\chi_1,\chi_2,\ldots,\chi_n$ be $n$ characters of $\mathbb{Q}_p^\times$ and $\pi\simeq I(\chi_1,\ldots,\chi_n)$ be the principle series representation of $\mathrm{GL}_n(\mathbb{Q}_p)$. We identify $\pi$ with the model of $\mathrm{Ind}_{B(\mathbb{Q}_p)}^{GL_n(\mathbb{Q}_p)}(\chi_1\otimes\cdots\otimes \chi_n)\cdot\delta_B)$.  Suppose $v\in\pi$ is the function on $B(\mathbb{Q}_p)w_\ell\Gamma_1(p)$ which is $1$ on $w_\ell$ (where $w_\ell$ denotes the longest Weyl element) and is invariant under the action of $N(\mathbb{Z}_p)$. Then $v$ is an eigenvector for the Hecke operators $U_{t_i}$ for $i=1,2,\ldots,n$ with Hecke eigenvalues
$$\chi_1\ldots\chi_i(p^{-1})p^{\kappa_1+\cdots+\kappa_i}.$$
\end{proposition}
\begin{proof}
See \cite[Lemma 4.4.2]{WAN}.
\end{proof}
Now suppose we have a weight $(a_1,\cdots,a_r,0;\kappa)$ for $U(r+1,1)$. Let $w_r=\begin{pmatrix}&1\\1_r&\end{pmatrix}\in M_{r+1,r+1}$.
\begin{proposition}\label{prop4.9}
Let the notation be as above. Let $P\subset \mathrm{GL}_{r+2}$ be the maximal parabolic consisting of matrices such that the below- or left-to-diagonal entries of the first column and of the last row are $0$. Suppose $\pi(\chi_1,\ldots,\chi_r)$ is an unramified principal series representation such that the $\chi_i$'s are pairwise distinct. Let $\tau_1,\tau_2$ be two characters of $\mathbb{Q}_p^\times$ with conductor $(p)$ such that $\mathrm{cond}(\tau_1\tau_2)=(p)$ as well. Consider $I(\tau_2^{-1}, \chi_1,\ldots,\chi_r,\tau_1)$, and identify it with the induced representation $\mathrm{Ind}_{P(\mathbb{Q}_p)}^{GL_{r+2}(\mathbb{Q}_p)}((\tau_2^{-1}\otimes \pi\otimes\tau_1)\cdot\delta_P)$. Then there is $v\in\pi$ a unique up to scalar stabilized vector such that the eigenvalues for $U_{t_i}$ are $a_{p,i}=\chi_1\cdots\chi_i(p^{-1})p^{\kappa_1+\cdots\kappa_i}$ for $i=1,\ldots,r$ such that if $v'\in I(\rho)$ is the function on $GL_{r+2}(\mathbb{Z}_p)$ supported in $P(\mathbb{Z}_p)w_r\Gamma_1(p)$ such that $v'(w_r)=v$ and $v'$ is invariant under $\Gamma_1(p)$, then $v'$ is an eigenvector for Hecke operators $U_{t_i}$, with eigenvalues $a_{p,i}$,\ldots,$a_{p,r}$, $a_{p,r}\tau_1(p)^{-1}p^{-\frac{r+\kappa}{2}}$, $a_{p,r}\tau_1(p)^{-1}\tau_2(p)p^{\kappa-r-1}$.
\end{proposition}
\begin{proof}
By assumption $(\tau_2^{-1},\chi_1,\cdots,\chi_r,\tau_1)$ is regular in the sense of Casselman. By re-ordering the characters properly the image of the vector provided by the last proposition under certain intertwining operator is a Hecke eigenvector with the eigenvalues given in the proposition. We prove it is the $v'$ described in the proposition. By checking the right action of $\Gamma_0(p)$ on $v'$, we see that it has to be supported in $P(\mathbb{Z}_p)w_r\Gamma_0(p)$. By checking the Hecke actions of $U_{p,i},\ldots,U_{p,r}$, we see that $v'(1)=v$ is a Hecke eigenvector with eigenvalues $a_{p,i}$. The uniqueness follows from the assumption that the $\chi_i$'s are pairwise distinct. The proposition follows.
\end{proof}

\begin{remark}
Note that we can allow critical slope, as long as it is finite.  We require that the level of $\varphi$ has to be at most $(p)$ (no deeper level), because it is otherwise difficult to prove that the pullback section is an eigenvector for the $U_p$-operators.
\end{remark}

\subsubsection{Pullback Formulas}

  Suppose $\tau_p=(\tau_1,\tau_2)$ is such that each of $\tau_1$, $\tau_2$, and  $\tau_1\tau_2$ has conductor $(p)$.  We let\footnote{In \cite{WAN}, $\xi_1^\dag$ is used to denote a product of $\tau_2$ times another character.   In the present situation, though, that character is trivial; so $\xi_1^\dag = \tau_2$.} $\xi_1^\dag=\tau_2$.  Definition \ref{expression} is a special case of part of \cite[Corollary 4.4.29]{WAN}.
\begin{definition}\label{expression}
We define
$$f_{v,\sieg}=f^0_p(s,g):=p^{-\sum_{i=1}^{r} i}\mathfrak{g}(\xi_1^\dag)^s
\times\sum_{A}\prod_{i=1}^r\bar\xi_1^\dag(\det A)\tilde{f}_1\left(s,g\begin{pmatrix}1_{r+1}&\begin{matrix}0&A\\0&0\end{matrix}\\&1_{r+1}\end{pmatrix}\right)$$
where $A$ runs through the set of $r\times r$ matrices
$$\begin{pmatrix}1&\cdots&m_{1r}\\&1&\cdots\\&&1\end{pmatrix}\begin{pmatrix}x_1&&\\&\ddots&\\&&x_r\end{pmatrix}\begin{pmatrix}1&&\\\cdots&1&\\n_{r1}&\cdots&1\end{pmatrix}$$
with $x_i\in p^{-t}\mathbb{Z}_p^\times\mathrm{mod}\ \mathbb{Z}_p$, $m_{ij},n_{ij}\in\mathbb{Z}_p/p\mathbb{Z}_p$.
We also define a Siegel section
$$f_{v,\sieg}'=f^{0'}_p(s,g)=p^{-\sum_{i=1}^{r} i}\mathfrak{g}(\xi_1^\dag)^a
\times\sum_{A}\prod_{i=1}^r\bar\xi_1^\dag(\det A)\tilde{f}_1\left(s,g\begin{pmatrix}1_{r}&A\\&1_{r}\end{pmatrix}\right)$$
for $A$ as above.
\end{definition}

We define elements $\Upsilon\in U(r+1,r+1)(\mathbb{Q}_p)$ and $\Upsilon'\in U(r,r)(\mathbb{Q}_p)$ such that $\Upsilon_{v_0}=S_{v_0}$ and $\Upsilon'_{v_0}=S_{v_0}'$.

\begin{proposition}\label{prop4.10}
Let the notation be as in Proposition \ref{prop4.9}. Let $\varphi\in\pi_p$ be a stabilized vector with Hecke eigenvalues $a_{p,i}$ for $U_{p,i}$ $(i=1,\ldots,r)$ provided in the last proposition. Then $F_\varphi(f^0;s,-)$ is the Klingen section supported in $P(\mathbb{Q}_p)w_r\Gamma_0(p)$ such that the right action of $\Gamma_0(p)$ is given by $\tau_2^{-1}(g_{r+1,r+1})\tau_1(g_{r+2,r+2})$ (where $g_{i, i}$ denotes the $i$-th diagonal entry of $g$), and such that
$$F_\varphi\left(f^0_p(s, -\Upsilon);s,w_r\right)=p^{\frac{\kappa r}{2}-\frac{r(r+1)}{2}}\mathfrak{g}(\tau_{1}^{-1})^r\prod_{i=1}^r(\chi_{i}\tau_{1})(p)\prod_{i=1}^r(\chi_{i}^{-1}\tau_{2})(p)
\bar{\tau}^c((p^r,1))\tau'(p^{-1})p^{\kappa-r}\mathfrak{g}(\bar{\tau}')^{-1}\varphi.$$
\end{proposition}
\begin{proof}
The proof is similar to the argument after \cite[Remark 4.4.2]{WAN}. The proof uses the trick of functional equations.
\end{proof}
  We define
\begin{align}\label{fpkling-equ}
f_{p,\Kling}(s, g):=F_\varphi\left(f^{0}_p;s,g\right).
\end{align}
It follows from the above two propositions that this is an eigenvector of $U_{t_i}$'s for each $i$ with non-zero eigenvalues. The following proposition can be proved similarly.

\begin{proposition}\label{prop411}
Let the notation be the same as in Proposition \ref{prop4.10}. Let $\varphi\in\pi_p^{\Gamma_1(p)}$ be a stabilized vector with Hecke eigenvalues $a_{p,i}$ for $U_{t_i}$ $(i=1,\ldots,r)$. Then
$$F_\varphi'(f^{0'}_p(s, -\Upsilon');s,1)=p^{\frac{\kappa r}{2}-\frac{r(r+1)}{2}}\mathfrak{g}(\tau_{1}^{-1})^r\prod_{i=1}^r(\chi_{i}\tau_{1})(p)\prod_{i=1}^r(\chi_{i}^{-1}\tau_{2})(p)
\bar{\tau}^c((p^r,1))\varphi.$$
\end{proposition}

  \subsubsection{Fourier Coefficients}

  We define the set $\mathfrak{X}=\mathfrak{X}_{\xi_1^\dag}$ to be the set of $r\times r$ matrices $x$ with $\mathbb{Z}_p$-coefficients such that $x_{11}\in\mathbb{Z}_p^\times$, $\det\begin{pmatrix}x_{11}&x_{12}\\x_{21}&x_{22}\end{pmatrix}\in\mathbb{Z}_p^\times$,\ldots,$\det x\in\mathbb{Z}_p^\times$, i.e. all the (determinants of the) upper left minors are in $\ZZ_p^\times$.  We define $\Phi_{\xi_1^\dag}$ to be the function on the space $M_{r\times r}(\ZZ_p)$ of $r\times r$ matrices with coefficients in $\ZZ_p$ such that $\Phi_{\xi_1^\dag}(x)=0$ if $x\not\in\mathfrak{X}$ and $\Phi_{\xi_1^\dag}(x)=\xi_1^\dag(\det x)$ if $x\in\mathfrak{X}$.

Proofs of the following two lemmas (Lemmas \ref{lemma4.12} and \ref{lemma4.13}) appear in \cite[Lemma 4.4.30]{WAN}.
\begin{lemma}\label{lemma4.12}
Suppose $\beta = \begin{pmatrix}\beta_{11}&\hdots& \beta_{1,r+1}\\
\vdots & \vdots & \vdots\\
\beta_{r+1, 1}&\hdots& \beta_{r+1, r+1}
\end{pmatrix}\in S_{r+1}\left(\IQ_p\right)$ is such that $\det\beta\not=0$. Let $\tilde{\beta}$ be the $r\times r$ matrix $\begin{pmatrix}\beta_{12}&\cdots&\beta_{1r+1}\\\cdots&\cdots&\cdots\\ \beta_{r2}&\cdots&\beta_{rr+1}\end{pmatrix}$.  If $\beta\not\in S_{r+1}(\mathbb{Z}_p)$, then $f^0_\beta(s,1)=0$. If $\beta\in S_{r+1}(\mathbb{Z}_p)$ then
$$f_\beta^0(s,1)=\bar{\tau}'(\det\beta)|\det\beta|_p^{2s}\mathfrak{g}(\tau')^{r+1}c_{r+1}(\bar{\tau}',-s)\Phi_{\xi_1^\dag}({}^t\!\tilde{\beta})$$
where $c_n(\tau',s)=\tau'(p^{n})p^{2ns-n(n+1)/2}$.
\end{lemma}\label{lemma4.13}

\begin{lemma}
Suppose $\beta\in S_{r\times r}(\mathbb{Q}_p)$ is such that $\det\beta\not=0$. Then if $\beta\not\in S_{r}(\mathbb{Z}_p)$ then $f^{0'}_\beta(s,1)=0$. If $\beta\in S_{r+1}(\mathbb{Z}_p)$ then
$$f_\beta^{0'}(s,1)=\bar{\tau}'(\det\beta)|\det\beta|_p^{2s}\mathfrak{g}(\tau')^{r}c_{r}(\bar{\tau}',-s)\Phi_{\xi_1^\dag}({}^t\!{\beta}).$$
\end{lemma}

\subsection{The sections at $\infty$}
We first define Siegel sections and give the associated Fourier coefficients (Sections \ref{ssections} and \ref{fcoef}, respectively).  Then we use pullback integrals for both the Klingen Eisenstein case and the $p$-adic $L$-functions cases (Sections \ref{pb1} and \ref{pb2}, respectively).
\subsubsection{Siegel Sections}\label{ssections}
We define $\mathbf{i} = \begin{pmatrix}i &\\&\frac{\zeta}{2}\end{pmatrix}$ and $\tilde{\mathbf{i}}:=i 1_{r+1}$, $\mathbf{i}'=\frac{\zeta}{2}$, $\tilde{\mathbf{i}}'=i1_r$.  We define auxiliary sections $\tilde{f}_\kappa\in I_{r+1}(\tau)$ and $\tilde{f}_\kappa'\in I_r(\tau)$.
The Siegel sections we choose are $\tilde{f}_\kappa(g,s):=J_{r+1}(g,\tilde{\mathbf{i}})^{-\kappa}|J_{r+1}(g,\tilde{\mathbf{i}})|^{\kappa-2s-r-1}$ and $\tilde{f}_\kappa'(g,s)=J_r(g, \tilde{\mathbf{i}})^{-\kappa}|J_r(g,\tilde{\mathbf{i}})|^{\kappa-2s-r}$.  We also define sections $f_\kappa$ and $f_\kappa'$ similarly to $\tilde{f}_\kappa$ and $\tilde{f}_\kappa'$ but with $\tilde{\mathbf{i}}$ replaced by $\mathbf{i}$.

\subsubsection{Fourier Coefficients}\label{fcoef}
\begin{lemma}\label{Fourier}
Suppose $\beta\in S_n(\mathbb{R})$. Then the function $\rightarrow \tilde{f}_{\kappa,\beta}(s,g)$ has a meromorphic continuation to all of $\mathbb{C}$.
Suppose $\kappa\geq n$.  Then  $\tilde{f}_{\kappa,\beta}(s,g)$ is holomorphic at $z_\kappa:=(\kappa-n)/2$.  Also, for $y\in \mathrm{GL}_n(\mathbb{C}),\tilde{f}_{\kappa,\beta}(z_\kappa,\mathrm{diag}(y,{}^t\!\bar{y}^{-1}))=0$ if $\det\beta\leq 0$; and if $\det\beta>0$, then
$$\tilde{f}_{\kappa,\beta}(z_\kappa,\mathrm{diag}(y,{}^t\!\bar{y}^{-1}))=\frac{(-2)^{-n}(2\pi i)^{n\kappa}(2/\pi)^{n(n-1)/2}}{\prod_{j=0}^{n-1}(\kappa-j-1)!}e(i\mathrm{Tr}(\beta y{}^t\!\bar{y}))\det(\beta)^{\kappa-n}\det\bar{y}^\kappa.$$
\end{lemma}
Taking $y$ to be the diagonal matrix such that the entry in the diagonal is a square root of the corresponding entry of $i^{-1}{\mathbf{i}}$, we obtain the Fourier coefficients of $f_\kappa$ and $f_\kappa'$ at the identity from those for $\tilde{f}_\kappa$ and $\tilde{f}_\kappa'$.\\

\subsubsection{Pullback integrals: the Klingen Eisenstein series case}\label{pb1}
  Let $\underline{k}=(a_1,\ldots,a_r)$ be a weight for $U(r,0)$ with $a_r\geq 0$, and let $\pi_\infty$ be the corresponding finite dimensional representation of $U(r,0)(\mathbb{R})$. Suppose $\tau_\infty$ is of infinity type $z\mapsto z^{-\frac{\kappa}{2}}\bar{z}^{\frac{\kappa}{2}}$, where $\kappa>r+1$ is a positive integer. Let $z_\kappa=\frac{\kappa}{2}-\frac{r+1}{2}$. Consider the standard representation
\begin{align*}
V_d:=(St_{\GL_{r+1}}\boxtimes St_{\GL_{r+1}})^d
\end{align*}
 of $GL_{r+1}\times GL_{r+1}$ and
 \begin{align*}
 V_{\kappa,d}=V_d\otimes (1\boxtimes \det)^\kappa.
 \end{align*}
 Recall that $d=2(a_1+\cdots+a_r)$. Note that \cite[Appendix]{Shi00} gives another interpretation of the $C^\infty$-differential operators $D_\kappa^d$ as a vector of elements in the enveloping algebra of $\mathfrak{u}(r+1,r+1)(\mathbb{R})$, which we denote by $D_\kappa^d$ as well.  Thus we can apply $D_\kappa^d$ to $f_\kappa\in I_{r+1}(\tau)$ to obtain an element $D_\kappa^d f_\kappa\in (I_{r+1}(\tau)\otimes V_{\kappa,d})^K$. Under the natural embedding
\begin{align}\label{emeddingq}
GL_{r+1}\times GL_1\times GL_r&\hookrightarrow GL_{r+1}\times GL_{r+1}\\
(\alpha, \beta, \gamma)&\mapsto \left(\alpha, \begin{pmatrix}\beta & 0\\ 0 & \gamma\end{pmatrix}\right),
\end{align}
the representation $L^{(\underline{k},0;\kappa)}\boxtimes (L^{(\underline{k})}\otimes\det^\kappa)$ (here $L^{(\underline{k},0;\kappa)}$ means $L^{(\underline{k},0)}\boxtimes L^\kappa$) shows up as a summand of $V_{\kappa,d}$. Let $f_\infty^{L^{(\underline{k},0;\kappa)}\boxtimes (L^{(\underline{k})}\otimes\det^\kappa)}$ be the $L^{(\underline{k},0;\kappa)}\boxtimes (L^{(\underline{k})}\otimes\det^\kappa)$-valued function obtained by pulling back $D_\kappa^df_\kappa$ to $U(r+1,1)\times U(0,r)$ by $\gamma$ and taking the summand corresponding to $L^{(\underline{k},0;\kappa)}\boxtimes (L^{(\underline{k})}\otimes\det^\kappa)$.

  We consider the representation $I(\rho_\infty)=I(\rho_\infty,z_\kappa)$ of $U(r+1,1)(\mathbb{R})$. As in \cite[1.4]{SU06},
$$(V_\pi^{\smfin}\otimes L^{\underline{k}})^{K^{(r,0)}_\infty}=(V_\pi^{\smfin}\otimes L^{(\underline{k},0;\kappa)})^{K^{(r,0)}_\infty}.$$
Here we use the superscript $(r,s)$ to denote that it is the maximal compact subgroup for $U(r,s)(\mathbb{R})$ and ``$\smfin$'' denotes the smooth vectors that are $K_\infty$-finite. Moreover, by Frobenius reciprocity we have a canonical isomorphism:
$$(I(\rho_\infty)\otimes L^{(\underline{k},0;\kappa)})^{K^{(r+1,1)}_\infty}\simeq (V_\pi^{\smfin}\otimes L^{(\underline{k},0;\kappa)})^{K^{(r,0)}_\infty}.$$
We let $\varphi_{(\underline{k},0;\kappa)}\in I(\rho_\infty)\otimes L^{(\underline{k},0;\kappa)}$ be the element corresponding to $\varphi\in (V_\pi^{sm,fin}\otimes L^{(\underline{k},0;\kappa)})^{K^{(r,0)}_\infty}$ under the above isomorphism.\\

\begin{definition}\label{csubscriptdef1}
Let $\varphi_{\infty}\in(\pi_\infty\otimes L^{(\underline{k})})^{U(\mathbb{R})}$. We define a $L^{(\underline{k},0;\kappa)}\boxtimes (L^{(\underline{k})}\otimes\det^\kappa)$-valued section
$$F_{\varphi_\infty}(g)=\int_{U(0,r)(\mathbb{R})}\langle f_{\infty}^{L^{(\underline{k},0;\kappa)}\boxtimes (L^{(\underline{k})}\otimes\det^\kappa)}(\gamma(g,g_1))\bar{\tau}(g_1),\pi(g_1)\varphi_{\infty}\rangle dg_1.$$
This is in $(I(\rho)\otimes (L^{(\underline{k},0;\kappa)}\boxtimes (L^{(\underline{k})}\otimes\det^\kappa)))^{K_\infty}$ and is a constant times $\varphi_{(\underline{k},0;\kappa)}$ defined above. We denote this constant by $c_{(\underline{k},0;\kappa)}$ and define $f_{v,Kling}:=F_{\varphi_\infty}$.
\end{definition}
\begin{lemma}\label{Harris}
Under the above situation, suppose the Klingen Eisenstein series is in the absolute convergence range for $P$. Then in part (ii) of Lemma \ref{constant}, we have $A(\rho,f,z_\kappa)_{-z_\kappa}=0$.
\end{lemma}

\begin{proof}
See \cite[Theorem 2.4.5]{HarrisESeries}.  Note that our section is exactly the one chosen in \emph{loc.cit}.
\end{proof}
(See \cite[(2.5.1.3)]{HarrisESeries} for details on the range of absolute convergence for the parabolic subgroup $P$.)

We are going to define a vector-valued Eisenstein series as follows. Let $v_1,\ldots,v_n$ be a basis for $L^{(\underline{k},0;\kappa)}$, and let $\varphi_{(\underline{k},0;\kappa)}=\sum_i\varphi_{(\underline{k},0,\kappa),i}\otimes v_i$. Let the vector-valued section $f=\otimes_vf_v\in I(\rho)\otimes L^{(\underline{k},0;\kappa)}$ be such that $f_\infty=\varphi_{(\underline{k},0;\kappa)}$. We define
\begin{align*}
E(f,s,g)=\sum_i E(f_i,s,g)\otimes v_i
\end{align*}
 and for $h\in   U(r+1,1)(\mathbb{A}_f)$, a classical (compared to adelic)
 \begin{align}
 E(f;s,Z,h)=\rho_{(\underline{k},0;\kappa)}(J(g,i))E(f;s,gh)
 \end{align}
  for $g\in U(r+1,1)(\mathbb{R})$ such that $g(i)=Z$. The following proposition can be proved in the same way as \cite[Page 480]{SU06}.
\begin{proposition}\label{questionprop}
Let the assumptions be as above. Then $E(f;s,Z,h)$ is a $L^{(\underline{k},0;\kappa)}$-valued holomorphic modular form.
\end{proposition}

\subsubsection{Pullback integrals: the $p$-adic $L$-functions case}\label{pb2}
  Let $z_\kappa'=\frac{\kappa}{2}-\frac{r}{2}$. As above, there is another interpretation of the $C^\infty$-differential operators $D_\kappa^d$ as a vector of elements in the enveloping algebra of $\mathfrak{u}(r,r)(\mathbb{R})$, which we denote by $D_\kappa^d$ as well. Thus we can apply $D_\kappa^d$ to $f_\kappa\in I_{r}(\tau)$ to obtain an element $D_\kappa^d f_\kappa\in (I_{r}(\tau)\otimes V_{\kappa,d})^K$. Consider the representations $V_d:=(St_{\GL_r}\boxtimes St_{\GL_r})^d$ of $GL_{r}\times GL_{r}$ and $V_{\kappa,d}:=V_d\otimes (1\boxtimes \det)^\kappa$.
The representation $L^{(\underline{k})}\boxtimes (L^{(\underline{k})}\otimes \det^\kappa)$ shows up as a summand of $V_{\kappa,d}$. Let $f_\infty^{L^{(\underline{k})}\boxtimes (L^{(\underline{k})}\otimes \det^\kappa)}$ be the $L^{(\underline{k})}\boxtimes (L^{(\underline{k})}\otimes \det^\kappa)$-valued function obtained by $D_\kappa^df_\kappa$ pulled back to $U(r,0)\times U(0,r)$ by $\gamma'$, and take the summand corresponding to $L^{(\underline{k})}\boxtimes (L^{(\underline{k})}\otimes \det^\kappa)$.
\begin{definition}\label{csubscriptdef2}
Let $\varphi_{\infty}\in(\pi_\infty\otimes V_{\underline{k}})^{U(\mathbb{R})}$. We define the $V_{(\underline{k})}$-valued section:
$$F_{\varphi_\infty}'(g)=\int_{U(0,r)(\mathbb{R})}\langle f_{\infty}^{L^{(\underline{k})}\boxtimes (L^{(\underline{k})}\otimes \det^\kappa)}(\gamma'(g,g_1))\bar{\tau}(g_1),\pi(g_1)\varphi_{\infty}\rangle dg_1.$$
This is in $(I(\rho)\otimes V_{(\underline{k})})^{K_\infty}$ and is a constant times $\varphi_{\infty}$.  We denote this constant by $c_{(\underline{k},0;\kappa)}'$.
\end{definition}
The constant $c_{(\underline{k},0;\kappa)}'$ is an algebraic number by \cite{Harris}.

\section{Global Computations}\label{Global-Computations-section}
\subsection{$p$-adic Interpolation}
We define an ``Eisenstein datum'' $\mathcal{D}$ to be a pair $(\varphi,\xi_0)$ consisting of a cuspidal eigenform $\varphi$ of weight $\underline{k}=(a_1,\cdots,a_r), a_r\geq 0$ on $U(r, 0)$ like in Section \ref{pbsection} and  a Hecke character $\xi_0$ of $\mathcal{K}^\times\backslash \mathbb{A}_\mathcal{K}^\times$ such that $\xi_0|\cdot|^{\frac{r-1}{2}}$ is a finite order character.  We denote by $\alpha$ the $\mathcal{O}_L$-isomorphism $\Lambda_{\mathcal{K},\mathcal{O}_L}^+\rightarrow\Lambda_{\mathcal{K},\mathcal{O}_L}^-$ sending $\gamma^+$ to $\gamma^-$. Let $\sigma$ be the reciprocity map of class field theory $\mathcal{K}^\times\backslash \mathbb{A}_\mathcal{K}^\times\rightarrow \mathrm{Gal}_\mathcal{K}^{ab}$ normalized by the geometric Frobenius. We define
\begin{align*}
\tau_0 &:=\overline{(\xi_0|\cdot|^{\frac{r-1}{2}})}^c,\\
\boldsymbol{\xi} &:=\xi_0\cdot(\Psi\circ\sigma),\\
 \boldsymbol{\tau} &:=\tau_0\cdot(\Psi^+\circ\sigma)\cdot(\alpha\circ\Psi^+\circ\sigma),\\
 \psi_\mathcal{K}&:=(\Psi^-)^{\frac{1}{2}}(\alpha\circ\Psi^+)^{-\frac{1}{2}}.
\end{align*}
We define\footnote{The superscript ``pb'' stands for ``pullback.''} $\mathcal{X}^{pb}$ to be the set of $\bar{\mathbb{Q}}_p$-points $\phi\in\mathrm{Spec}\Lambda_{\mathcal{K},\mathcal{O}_L}$ such that $\phi\circ\boldsymbol{\tau}((1+p,1))=\tau_0((1+p,1))$,
\begin{align}\label{kappaphidefn}
\phi\circ\boldsymbol{\tau}((1,1+p))=(1+p)^{\kappa_\phi}\tau_0((1,1+p))
\end{align}
 for some integer $\kappa_\phi>r+1$ such that the weight $(c_r,\cdots,c_1,0;\kappa_\phi)$ is in the absolutely convergent range for $P$ in the sense of Harris (lemma \ref{Harris}), and such that
 \begin{align}\label{mphidefn}
 \phi\circ\psi_{\CK}(\gamma^-)=(1+p)^{\frac{m_\phi}{2}}
 \end{align}
  for some non-negative integer $m_\phi$, and such that the $\tau_\phi$ (to be defined in a moment) is such that, under the identification $\tau_\phi=(\tau_1,\tau_2)$ for $\mathcal{K}_p^\times\simeq \mathbb{Q}_p^\times\times\mathbb{Q}_p^\times$, we have $\tau_1$, $\tau_2$, $\tau_1\tau_2$ all have conductor $(p)$.

We denote by $\mathcal{X}$ the set of $\bar{\mathbb{Q}}_p$-points $\phi$ in $\mathrm{Spec}\Lambda_{\mathcal{K},\mathcal{O}_L}$ such that $\phi\circ\boldsymbol\tau((1,1+p))=(1+p)^{\kappa_\phi}\zeta_1\tau_0((1,1+p))$, $\phi\circ\boldsymbol{\tau}((p+1,1))=\tau_0((p+1,1))$, and $\phi\circ\psi_{\CK}(\gamma^-)=\zeta_2$ with $\zeta_1$ and $\zeta_2$ $p$-power roots of unity.
\begin{remark}
We will use the points in $\mathcal{X}^{pb}$ for $p$-adic interpolation of special $L$-values and Klingen Eisenstein series, and we will use the points in $\mathcal{X}$ to construct a Siegel Eisenstein measure.
\end{remark}

For each $\phi\in\Spec \Lambda_{\mathcal{K}, \mathcal{O}_L}$, we define Hecke characters $\psi_\phi$ and $\tau_\phi$ of $\mathcal{K}^\times \backslash\mathbb{A}_\mathcal{K}^\times$ by
\begin{align*}
\bar{\tau}_\phi^c(x)&:=\bar{x}_\infty^{\kappa_\phi}(\phi\circ\boldsymbol{\tau})(x)x_{\bar{v}}^{-\kappa_\phi}\cdot|\cdot|^{-\frac{\kappa_\phi}{2}},\\
\psi_\phi(x)&:=x_\infty^{\frac{m_\phi}{2}}\bar{x}_\infty^{-\frac{m_\phi}{2}}(\phi\circ\boldsymbol{\psi})x_v^{-\frac{m_\phi}{2}}x_{\bar{v}}^
{\frac{m_\phi}{2}},
\end{align*}
where $\kappa_\phi$ and $m_\phi$ are as in Equations \eqref{kappaphidefn} and \eqref{mphidefn}, respectively.
Let
\begin{align*}
\xi_\phi&=|\cdot|^{\frac{\kappa-r+1}{2}}\bar{\tau}_\phi^c\psi_\phi\psi_\phi^{-c},\\
\varphi_\phi&=\varphi\otimes\psi_\phi^{-1}.
\end{align*}
\begin{definition}
The element $\phi$ considered as a function on $\Lambda_\mathcal{K}$ restricts to a character of $\Gamma_\mathcal{K}$ which we denote by $\hat{\phi}$.
\end{definition}
We are going to construct $p$-adic families of modular forms that interpolate the Klingen Eisenstein series, and we are going to construct $p$-adic $L$-functions from the datum $(\varphi_\phi,\tau_\phi).$ (The families are parametrized by weights $\underline{k}_\phi$, where $\underline{k}_\phi=(a_1+m_\phi,\ldots,a_r+m_\phi)$.) For an arithmetic point $\phi\in\mathcal{X}$, recall that we have defined
$$f_{\sieg,\phi}=\prod_{i=0}^{r}L^\Sigma(2s+r+1-i,\bar\tau\chi_\mathcal{K}^i)\prod_{j=1}^r(\kappa-j-1)!(2\pi i)^{-(r+1)\kappa}\left(\frac{2}{\pi}\right)^{-\frac{(r+1)r}{2}}\otimes_{v\not=\infty}f_{v,\sieg}\otimes_\infty f_\kappa$$
$$f_{\sieg,\phi}'=\prod_{i=0}^{r-1}L(2s+r-i,\bar\tau\chi_\mathcal{K}^i)\prod_{j=0}^{r-1}(\kappa-j-1)!(2\pi i)^{-r\kappa}\left(\frac{2}{\pi}\right)^{\frac{-r(r+1)}{2}}\otimes_{v\not=\infty}f_{v,\sieg}'\otimes_\infty f_\kappa',$$
where $f_{p,\phi}$ (respectively, $f_{p,\phi}'$) is the Siegel section constructed in Section \ref{local-computations-section}, using the datum $\varphi_\phi,\tau_\phi$.
We define
$$f_{\Kling,\phi}=\prod_{i=0}^{r}L(2s+r+1-i,\bar\tau\chi_\mathcal{K}^i)\otimes_vf_{v,\Kling}$$
for $f_{v,\Kling}$ ($v\not=p,\infty$) as defined in Section \ref{local-computations-section} using the datum $(\varphi_\phi,\tau_\phi)$. (notation as in definition \ref{csubscriptdef1}).

\subsection{Action of the Differential Operators on $q$-expansion coefficients}
\subsubsection{The Klingen Eisenstein Series Case}
Recall that under the embedding $GL_{r+1}\times GL_1\times GL_r\hookrightarrow GL_{r+1}\times GL_{r+1}$ given in Equation \eqref{emeddingq}, $L^{(a_1,\ldots,a_r,0)}\times L^\kappa\times L^{(\kappa+a_1,\ldots,\kappa+a_r)}$ (viewed as a representation of $GL_{r+1}\times GL_1\times GL_{r}$) is a summand of $(St_{\GL_{r+1}}\boxtimes St_{\GL_{r+1}})^d\otimes (1\boxtimes \det_{r+1})^\kappa$ (for some $d$) restricted to $GL_{r+1}\times GL_1\times GL_r\hookrightarrow GL_{r+1}\times GL_{r+1}$. We choose such a summand and let $v_{(\underline{k},0,\kappa)}$ be the highest weight vector of it (viewed as an element of the representation of $GL_{r+1}\times GL_{r+1}$).  More precisely, if $F_\kappa$ is a form of scalar weight $\kappa$ on $U(r+1, r+1)$, we consider $D_\kappa^d F_\kappa$ as a $p$-adic automorphic form on $U(r+1,r+1)$.  We define an automorphic form $\hat{F}_{(\underline{k},0,\kappa)}$ on $U(r+1, r+1)$ by
$$\hat{F}_{(\underline{k},0,\kappa)}:=\langle D_\kappa^dF_\kappa,v_{(\underline{k},0,\kappa)}\rangle,$$
where $\langle,\rangle$ is the natural pairing between $V_{\kappa,d}^\vee$ and $V_{\kappa,d}$.  We will compute the $q$-expansion coefficients of $\hat{F}_{(\underline{k},0;\kappa)}$.
\begin{proposition}\label{qexpprop}
Let $F_\kappa$ be an automorphic form of scalar weight $\kappa$ on $U(r+1,r+1)$ with Fourier expansion $F_{\kappa,[g]}=\sum_{\beta\in S_{r+1}^+(\IQ)}F_{\beta,[g]} q^\beta$ at the cusp $[g]$. Then the $q$-expansion coefficient of $\hat{F}_{(\underline{k},0;\kappa)}$ at $\beta$ is given by
\begin{align}\label{fcoeffq}
F_{(\underline{k},0;\kappa),[g],\beta}=F_{\beta,[g]}\beta_{21}^{a_1-a_2}\det\begin{pmatrix}\beta_{21}&\beta_{22}\\ \beta_{31}&\beta_{32}\end{pmatrix}^{a_2-a_3}\cdots\det\begin{pmatrix}\beta_{21}&\cdots&\beta_{2r}\\\vdots&\ddots&\vdots\\ \beta_{r+1,2}&\cdots&\beta_{r+1,r}\end{pmatrix}^{a_r}.
\end{align}
\end{proposition}
\begin{proof}
Note that if we identify the representation $\mathrm{St}_{\GL_r}\boxtimes\mathrm{St}_{\GL_r}$ of $GL_{r}\times GL_r$ with the natural representation of it on $M_{r\times r}$ (i.e. $\alpha\mapsto { }^tg\alpha h$ for each $(g, h)\in \gl_r\times \gl_r$), then we have an action of $GL_r\times GL_r$ on the space of polynomial functions of entries of $M_{r\times r}$. There is a summand of this space with representation $L^{(a_1,\ldots,a_r)}\boxtimes L^{(a_1,\ldots,a_r)}$ whose highest weight vector is $x_{11}^{a_1-a_2}\det \begin{pmatrix}x_{11}&x_{12}\\x_{21}&x_{22}\end{pmatrix}^{a_2-a_3}\cdots(\det X)^{a_r}$. Let $(v_1,\ldots,v_{r+1})$ be a standard basis of the standard representation of $GL_{r+1}$. Note that we have used the notation of \cite{Hsieh}, and $\omega^+$ in \cite{Hsieh} is actually denoted $\omega^-$ in \cite{Eischen}.  (See the first paragraph in Section \ref{2.5}.) Thus, when applying the formula in \emph{loc.cit}, we need to give the transpose of each matrix $\beta$ appearing in the $q$-expansion coefficients.\\

  This is a straightforward application of \cite[Theorem 9.2(4)]{Eischen}. More precisely, we let
\begin{align}\label{dotequq}
v=(v_2\boxtimes v_1)^{a_1-a_2}\cdot ((v_2\boxtimes v_1)\cdot(v_3\boxtimes v_2)-(v_2\boxtimes v_2)\cdot(v_1\boxtimes v_3))^{a_2-a_3}\cdots.
\end{align}
 By \emph{loc.cit}, we have $DF=(\sum \beta_{ij}c(F,\beta)v_j^\vee\boxtimes v_i^\vee)q^\beta$. By definition, the resulting form is given by $\langle D^d_\kappa F,v\rangle$. Thus the $q$-expansion is as in the statement of Proposition \ref{qexpprop}.
\end{proof}
\subsubsection{The $p$-adic $L$-functions Case}
Recall that as a representation of $GL_{r}\times GL_{r}$, $L^{(a_1,\ldots,a_r)}\times L^{(\kappa+a_1,\ldots,\kappa+a_r)}$ is a summand of $(St_{\GL_r}\boxtimes St_{\GL_r})^d\otimes (1\boxtimes \det_r)^\kappa$ (for some $d$). We take such a summand and let $v_{(\underline{k},0;\kappa)}'$ be the highest weight vector.  Let $F'_\kappa$ be a form on $U(r,r)$ of scalar weight $\kappa$.  We define
\begin{align}\label{Epairingequ}
\hat{F}_{(\underline{k},0;\kappa)}'=\langle D_\kappa^dF_\kappa',v_{(\underline{k},0;\kappa)}'\rangle.
\end{align}
 We will compute the $q$-expansion coefficients of $\hat{F}_{(\underline{k},0;\kappa)}'$.
\begin{proposition}
Let $F_\kappa'$ be an automorphic form of scalar weight $\kappa$ on $U(r,r)$ with Fourier expansion $F'_\kappa=\sum_{\beta\in S_{r}^+(\beta)}F_{\beta,[g]}' q^\beta$ at the cusp $[g]$. Then the $q$-expansion coefficient at $\beta$ of the form $\hat{F}_{(\underline{k},0,\kappa)}'$ is given by
\begin{align}\label{expq2}
\hat{F}_{(\underline{k},0;\kappa),[g],\beta}'=F_{\beta,[g]}'\beta_{11}^{a_1-a_2}\det\begin{pmatrix}\beta_{11}&\beta_{12}\\ \beta_{21}&\beta_{22}\end{pmatrix}^{a_2-a_3}\cdots\det\begin{pmatrix}\beta_{11}&\cdots&\beta_{1 r}\\\vdots&\ddots&\vdots\\ \beta_{r1}&\cdots&\beta_{rr}\end{pmatrix}^{a_r}.
\end{align}
\end{proposition}
The proof is similar to the proof of Proposition \ref{qexpprop}.

\subsection{Construction of the Eisenstein Measure}
We start with Siegel Eisenstein series
\begin{align*}
E_\kappa&:=E_\kappa(f_{\sieg,\phi},z_\kappa,-)\\
E_\kappa'&:=E_\kappa'(f_{\sieg,\phi}',z_\kappa',-).
\end{align*}
We apply the differential operator constructed above to this Eisenstein series, pair with the vector $v_{(\underline{k},0;\kappa)}$ or $v_{(\underline{k},0;\kappa)}'$ and denote the resulting form by $\hat{E}_{\mathcal{D},\phi}(f_{\sieg,\phi},z_{\kappa_\phi},-)$ (resp. $\hat{E}_{\mathcal{D},\phi}'(f_{\sieg,\phi}',z_{\kappa_\phi},-)$).

\begin{proposition}\label{prop5.4}
There are $p$-adic measures $\mathcal{E}_{\mathcal{D},\sieg}$ and $\mathcal{E}_{\mathcal{D},\sieg}'$ on $\Gamma_\mathcal{K}\simeq \mathbb{Z}_p\times\mathbb{Z}_p$ with values in the space of $p$-adic automorphic forms on $U(r+1,r+1)$ and $U(r,r)$, respectively, such that
$$\int_{\Gamma_\CK}\hat{\phi}d\mathcal{E}_{\mathcal{D},\sieg}=\hat{E}_{\mathcal{D},\phi}
(f_{\sieg,\phi},z_{\kappa_\phi},-)$$
and
$$\int_{\Gamma_\CK}\hat{\phi}d\mathcal{E}_{\mathcal{D},\sieg}'=\hat{E}_{\mathcal{D},\phi}'
(f_{\sieg,\phi}',z_{\kappa_\phi}',-),$$
respectively.
\end{proposition}
\begin{proof}
The idea of the proof is the same as the main idea in the constructions of the Eisenstein measure in \cite[Section 4.2]{Katz78} and \cite[Section 4]{apptoHLS}.
More precisely, this proposition follows from the formulas for the $q$-expansion coefficients, together with the $p$-adic $q$-expansion principle for automorphic forms on unitary groups of signature $(r+1, r+1)$ (respectively, $(r, r)$).  (The $p$-adic $q$-expansion principle is given in \cite[Corollary 10.4]{hi05} and\cite[Section 8.4]{Hida04}.)
\end{proof}

\subsection{Construction of the Family}
In this section, we will prove parts (i) and (ii) of Theorem \ref{main}. Without loss of generality, we may take $K=\prod_vK_v\subset U(r,0)(\mathbb{A}_f)$ open compact subgroups such that $\mathcal{E}_{\mathcal{D},\sieg}$ and $\mathcal{E}_{\mathcal{D},\sieg}'$ are invariant under the action of $\gamma(1\times K)$.

  \subsubsection{Klingen Eisenstein Series}\label{section541}

In this section, we construct a bounded measure on $\mathbb{Z}_p^2$ with values in $M_{\underline{k}}^{(r+1,1)}(K_0(p),\mathcal{O}_L)$ that interpolates the Klingen Eisenstein series constructed before. Recall that we constructed a measure of Siegel Eisenstein series $\mathcal{E}_{\mathcal{D}, \sieg}$ in Proposition \ref{prop5.4}.  We consider the $V_{\infty,\infty}^{(r+1,1)}\otimes V_{\infty,\infty}^{(0,r)}$-valued measure $\tilde{\mathcal{E}}_\mathcal{D}$ on $\Gamma_\CK$ defined by
\begin{align*}
\tilde{\mathcal{E}}_\mathcal{D} = \boldsymbol{\tau}^{-1}(\det g_1)\boldsymbol{\psi}^{-1}(\det g_1)(\mathcal{E}_{\mathcal{D}, \sieg}\circ i)
\end{align*}
where $i$ is the embedding defined in Section \ref{Igusaunitary}. (Here, we write $(g,g_1)$ for elements in $U(r+1,1)\times U(0,r)$.).

Consider the action of group $1\times U(0,r)$. One first observes that at any arithmetic point $\phi\in\mathcal{X}$ the Eisenstein series constructed restricting to $U(0,r)$ is of weight $\underline{k}$ and invariant under the level $K_0(p)\subset U(0,r)(\mathbb{Q}_p)$ which are fixed throughout. Recall the discussion of subsection \ref{2.12}.  We can obtain a measure of $L^{(-c_r,\cdots, -c_1)}$-valued form on $U(r+1,1)\times U(0,r)$ which we still denote as $\tilde{\mathcal{E}}_\mathcal{D}$. The resulting forms on $U(0,r)$ will live in the $\hat{M}^{(0,r)}(K_0(p), \mathcal{O}_L[[\Gamma_\mathcal{K}]]\otimes L)$. (This is the very reason why our family will have coefficients in the Iwasawa algebra instead  of general affinoid Tate algebra). This is a bounded measure on $\Gamma_\mathcal{K}$.  This can be seen as follows.  Take a representative $\{g_1,\cdots,g_t\}$ of $G(\mathbb{Q})\backslash G(\mathbb{A}_f)/K^{(p)}{}^t\!K_0(p)$. Take a basis for $L^{(-c_r,\cdots, -c_1)}$ and write the entries of the vector valued $\tilde{\mathcal{E}}_\mathcal{D}$ considered here with respect to this fixed basis as $\tilde{\mathcal{E}}_{\mathcal{D},j}$'s for $j=1,2,\cdots$. Recall the scalar valued $\tilde{\mathcal{E}}_\mathcal{D}$ considered in the last paragraph takes $p$-adically integral values. Moreover it is a polynomial function when restricting to $g_iN(p\mathbb{Z}_p)$ of the entries in $N(p\mathbb{Z}_p)$ and degree determined by $\underline{k}$. Each $\tilde{\mathcal{E}}_{\mathcal{D},j}$ (which is actually some coefficient of that polynomial function by the construction of $L^{(-c_r,\cdots, -c_1)}$) is a $\mathbb{Q}_p$-linear combination of the values of the scalar valued $\tilde{\mathcal{E}}_\mathcal{D}$ at certain points in $g_iN(p\mathbb{Z}_p)$ (the coefficients in $\mathbb{Q}_p$ only depend on $\underline{k}$ and the basis of $L^{(-c_r,\cdots,-c_1)}$ we fixed). Therefore the $\tilde{\mathcal{E}}_{\mathcal{D},j}$'s must have bounded $p$-adic norms. 

We define a Klingen Eisenstein measure $\mathcal{E}_{\mathcal{D}, \Kling}$ by
$$\mathcal{E}_{\mathcal{D}, \Kling}(\phi):=\langle\tilde{\mathcal{E}}_\mathcal{D}(\phi),\hat\varphi\rangle_{\low}$$
for all continuous functions $\phi$ on $\Gamma_\mathcal{K}$, where the subscript ``$\low$'' means the pairing with respect to the group $\gamma(1\times U(0,r))$ (identifying $U(0, r)$ with $U(r, 0)$).  So $\mathcal{E}_{\mathcal{D}, \Kling}(\phi)$ takes values in the space of $p$-adic automorphic forms on $U(r+1, 1)$.
\begin{definition}
We define $\mathbf{E}_{\varphi,\xi_0}$ to be the Fourier-Jacobi expansion obtained by composing the measure $\mathcal{E}_{\mathcal{D}, \Kling}$ with the Fourier-Jacobi map.
\end{definition}

\begin{remark}\label{panchrmk}
For the reader familiar with \cite{Pan}, we briefly remark about the relationship of that work with ours in the case $r=2$.
In \cite{Pan} the level of the Eisenstein series at $p$ increases with the conductor of the character by which it is twisted. Thus he needs to construct the projector as $(U^{-\nu}\pi_{\alpha,1}U^\nu)$ (notation as in \emph{loc.cit}) where $\nu$ is the level and makes the Eisenstein measure only ``$h$-admissible'' instead of bounded as compared to our situation. In fact our situation coincides with the result obtained from another construction generalizing \cite{Hida} which constructed the $p$-adic $L$-functions for Rankin-Selberg convolutions for two Hida families $\mathbf{f}$ and $\mathbf{g}$, but we allow the family $\mathbf{g}$ with lower weight than $\mathbf{f}$ to be finite slope and $\mathbf{f}$ is still required to be ordinary.  (The construction works in the same way.)
\end{remark}
\subsubsection{Comparison of the $p$-adic and $C^\infty$ Differential Constructions}

  At a weight $(a_1,\ldots,a_r,0;\kappa)$, we consider
  \begin{align}\label{offequ}
E_{(\underline{k},0;\kappa)}^{C^\infty}(f_{\sieg},z_\kappa,-):=\mathrm{proj}(L^{(a_1,\cdots,a_r,0)}\boxtimes L^\kappa\boxtimes L^{(\kappa+a_1,\cdots,\kappa+a_r)})\left(\left(\partial(\kappa,C^\infty,d)E_\kappa(f_{\sieg},z_\kappa,-)\right)
\circ i\right).
\end{align}
where $i$ is the embedding defined in Section \ref{Igusaunitary}. (Here, by $\mathrm{proj}(L^{(a_1,\cdots,a_r,0)}\boxtimes L^\kappa\boxtimes L^{(\kappa+a_1,\cdots,\kappa+a_r)})$ we mean projection to a direct summand of $(\mathrm{St}_{r+1}\boxtimes \mathrm{St}_{r+1})^d$ isomorphic to $(L^{(a_1,\cdots,a_r,0)}\boxtimes L^\kappa\boxtimes L^{(\kappa+a_1,\cdots,\kappa+a_r})$.) Again by the discussion of \cite[Section 2.8]{Hsieh} on periods we know that this is
\begin{align*}
&E_{(\underline{k},0;\kappa)}^{C^\infty}(f_{\sieg},z_\kappa,-)(g,g_1)&&:=
\frac{1}{\Omega_\infty^{\frac{d}{2}+r\kappa}}\mathrm{proj}(L^{(a_1,\cdots,a_r,0)}\boxtimes L^\kappa\boxtimes L^{(\kappa+a_1,\cdots,\kappa+a_r)})&\\
&&&\left(\left(\partial(\kappa,C^\infty,d)E_\kappa(f_{\sieg},z_\kappa,-)\right)
\right)(\gamma(g,g_1)\cdot\Upsilon).&
\end{align*}

This is a nearly holomorphic form on $U(r+1,1)\times U(r)$. We can write $E_{(\underline{k},0;\kappa)}^{C^\infty}=\sum_i E_i\boxtimes \varphi_i$ where each $\varphi_i$ is a form in a certain irreducible automorphic representation of $U(0,r)(\mathbb{A}_\mathbb{Q})$. Thus, by Proposition \ref{questionprop}, we see that $E_{(\underline{k},0;\kappa)}^{C^\infty}(f_{\sieg},z_\kappa,-)$ is in fact holomorphic. (Although in our discussion for Klingen Eisenstein series we assumed the form we start with is tempered. However this is only for convenience of explicit calculation for pullback formula and is by no means serious. As long as the datum is in the absolutely convergence range of $P$ the whole argument works.)

  Let $\phi\in\mathcal{X}^{pb}$. Then by Proposition \ref{pullbackidentities-prop},
$$\left\langle E^{C^\infty}_{(\underline{k},0;\kappa)},\varphi\right\rangle_{\low}= \frac{1}{\Omega_\infty^{\frac{d}{2}+r\kappa}}E\left(f_{\Kling},z_\kappa,-\right)
.$$

In view of the pullback formula and our remarks in Section \ref{2.12}
$$\int_{\Gamma_\mathcal{K}}\phi d\mathcal{E}_{\mathcal{D}, \Kling}=\beta_{\underline{k}_\phi}
((\frac{1}{\Omega_\infty})^{\frac{d_\phi}{2}+r\kappa_\phi}E(f_{\Kling,\phi},z_\kappa,-))
.$$

  We can also consider
  \begin{align*}
  E_{(\underline{k},0;\kappa)}^{p-\mathrm{adic}}(f_{\sieg},z_\kappa,-):=\mathrm{proj}(L^{\underline{k},0}\boxtimes L^{\kappa+\underline{k}})\partial(\kappa,p-\mathrm{adic},d)E_\kappa(f_{\sieg},z_\kappa,-)\circ i.
  \end{align*}
We claim that the $p$-adic and $C^\infty$ constructions $E_{(\underline{k},0;\kappa)}^{C^\infty}$ and $E_{(\underline{k},0;\kappa)}^{p-\mathrm{adic}}$ actually coincide.
Consider the representation $\left(\mathrm{St}_{\GL_{r+2}}\times\mathrm{St}_{\GL_r}\right)^{d+\kappa r}$ of $\mathrm{GL}_{r+2}\times\mathrm{GL}_r$. Under the embedding $\mathrm{GL}_{r+1}\times\mathrm{GL}_1\times \mathrm{GL}_r\hookrightarrow \mathrm{GL}_{r+2}\times \mathrm{GL}_r$ it contains the representation $L^{(a_1,\ldots,a_r,0)}\boxtimes L^\kappa\boxtimes L^{(\kappa+a_1,\ldots,\kappa+a_r)}$ as a summand of representation of $\mathrm{GL}_{r+1}\times\mathrm{GL}_1\times\mathrm{GL}_r$. Note that the vector $v_{(\underline{k},0,\kappa)}$ is also the highest weight vector of an irreducible summand representation of $\left(\mathrm{St}_{\GL_{r+2}}\times\mathrm{St}_{\GL_r}\right)^{d+\kappa r}$. We write $E$ here for the Siegel Eisenstein series of weight $\kappa$ at the arithmetic point $\phi$. We consider the restriction of $D^dE$ to the group $U(r+1,1)\times U(0,r)$.  Let $G$ be the form corresponding to the irreducible sub-representation generated by $v_{(\underline{k},0;\kappa)}$ above. This can be viewed as a section of the deRham cohomology of the universal family.
So,  adapting the proof of \cite[Proposition 2.2.3]{urban} to the case of unitary groups by applying the description of the Gauss-Manin connection in terms of coordinates given in \cite[Section 3]{Eischen}, we see that $G$ itself takes values in the first filtration of the Hodge filtration.  More precisely, for any positive integer $d$, let $\mathcal{H}_k^d :=\omega^{\otimes k-d}\otimes Sym^d\left(\mathcal{H}_{dR}^{+}\otimes\mathcal{H}_{dR}^{-}\right)$, and let $\eta\in H^0\left(\mathcal{M}, \mathcal{H}_k^d\right)$.  Let $\pi: \mathcal{H}_n\rightarrow \mathcal{M}$ be the canonical projection.  As explained in \cite[Equations (3.6)-(3.10)]{Eischen}, there is a basis of $\mathbf{R}$-linear global relative one-forms $\alpha_i, \beta_i, \alpha_i', \beta_i'$, $1\leq i\leq n$, that are horizontal for the Gauss-Manin connection.  We also have the global relative one-forms $du_1, \ldots, du_{2n}, d\bar{u}_1, \ldots, d\bar{u}_{2n}$, where $u_1 \ldots, u_{2n}$ denote the standard coordinates in $\IC^{2n}$.  As explained in \cite[Equations (3.17)-(3.18)]{Eischen}, for $z\in\mathcal{H}_n$,
\begin{align}
\begin{pmatrix}
\beta_1+\alpha \beta_1'\\
\vdots\\
\beta_n+\alpha\beta_n'
\end{pmatrix}
&=\left({ }^tz-\bar{z}\right)^{-1}\begin{pmatrix}du_{n+1}-d\bar{u}_n\\ \vdots\\ du_{2n}-d\bar{u}_n\end{pmatrix}\label{uibeta-equ}\\
\begin{pmatrix}
\beta_1+\bar{\alpha} \beta_1'\\
\vdots\\
\beta_n+\bar{\alpha}\beta_n'
\end{pmatrix}&=\left(z-{ }^t\bar{z}\right)^{-1}\begin{pmatrix}du_1-d\bar{u}_{n+1}\\
\vdots\\
du_n-d\bar{u}_{2n}
\end{pmatrix}\label{uibeta-equ2},
\end{align}
where $\alpha$ is a certain purely imaginary complex number.  Now, let $\eta\in H^0\left(\mathcal{M}, \mathcal{H}_k^d\right)$.  Then $\pi^*\eta(z)$ is a linear combination of terms of the form $f(z)v$, where $f$ is holomorphic function on $\mathcal{H}_n$ and $v$ is a tensor product of $du_i$'s and $\beta_j\pm \alpha \beta_j'$'s ($1\leq i, j \leq n$).  Replacing each $\beta_i\pm\alpha\beta_i'$ by the appropriate term from the right hand side of Equations \eqref{uibeta-equ}-\eqref{uibeta-equ2}, we then see that $\eta$ can only be holomorphic if $v$ was only a tensor product of $du_i$'s, i.e. only if $r = 0$.  (When $n=1$, this is similar to the bottom of \cite[p. 7]{urban}, with $\tau$ replaced by $z$ and $\beta$ replaced by $\beta_1$.)  Thus the constructions using the unit root splitting and using the $C^\infty$-splitting are the same.  (Above, we have used \cite[Section 3]{Eischen} to adapt \cite[Proposition]{urban} to the case of $U(n,n)$; as noted on \cite[p. 4]{Eischen}, though, the constructions in \cite{Eischen} generalize to $U(r, s)$ for all $r\geq s>0$.  So we get a similar result for any non-definite unitary group $U(r, s)$, even if $r\neq s$.)\\

  \subsubsection{$p$-adic $L$-Functions}

  As in the construction of the Klingen Eisenstein family in Section \ref{section541}, by using the pullback $\gamma':U(r,0)\times U(0,r)\hookrightarrow U(r,r)$ and using $\mathcal{E}_{\mathcal{D}, \sieg}'$ in place of $\mathcal{E}_{\mathcal{D}, \sieg}$ and $\Upsilon'$ in the place of $\Upsilon$, we obtain a bounded measure $\tilde{\mathcal{E}}_\mathcal{D}'$ on $\Gamma_{\CK}$ that takes values in $V_{\infty,\infty}^{(r,0)}\otimes V_{\underline{k}}^{(0,r)}$. We similarly define a vector-valued measure $\tilde{\mathcal{E}}_\mathcal{D}$. We define $\mathcal{L}_{\varphi,\xi_0}^\Sigma\in\Lambda_{\mathcal{K},\mathcal{O}_L}$ to be such that
  \begin{align*}
\langle\tilde{\mathcal{E}}_\mathcal{D}',\hat{\varphi}\rangle_{\low}=\mathcal{L}_{\varphi,\xi_0}^\Sigma\varphi.
\end{align*}
By Propositions \ref{pullbackidentities-prop}(i) and \ref{prop411}, Lemmas \ref{lemma4.2} and \ref{lemma4.5}, and Definition \ref{csubscriptdef2}, we have the following interpolation formula for the $p$-adic $L$-function
$$\phi(\mathcal{L}_{\varphi,\xi_0}^\Sigma)=C_{\varphi,p}.\frac{L^\Sigma(\tilde{\pi},\xi_\phi,0)}{\Omega_\mathcal{K}^{d_\phi+r\kappa_\phi}}
c_{\phi}'.p^{\frac{\kappa_\phi r}{2}-\frac{r(r+1)}{2}}\mathfrak{g}(\tau_{1,\phi}^{-1})^r\prod_{i=1}^r(\chi_{i,\phi}\tau_{1,\phi})(p)\prod_{i=1}^r(\chi_{i,\phi}^{-1}\tau_{2,\phi})(p)
\bar{\tau}_\phi^c((p^r,1))$$
where $d_\phi=2(a_{1,\phi}+\cdots+a_{r,\phi})$, $c'_\phi=c'_{(\underline{k}_\phi,0,\kappa_\phi)}$ is an algebraic constant coming from an Archimedean integral and $C_{\varphi,p}$ is a product of local constant coming from the pullback integrals.

\subsection{Constant Terms}\label{constterms-section}
\subsubsection{$p$-adic $L$-functions for Dirichlet Characters}
As explained in \cite{KL} and \cite[Section 3.4.3]{SU}, there is an element $\mathcal{L}_{\bar{\tau}'}^\Sigma$ in $\Lambda_{\mathcal{K},\mathcal{O}_L}$ such that at each arithmetic point $\phi\in\mathcal{X}^{pb}$, $\phi(\mathcal{L}_{\bar{\tau}'})=L(\bar{\tau}'_\phi,\kappa_\phi-r)\cdot\pi^{r-\kappa_\phi}\tau_\phi'(p^{-1})p^{\kappa_\phi-r}\mathfrak{g}(\bar{\tau}_\phi')^{-1}$.   (This element is a $p$-adic Dirichlet $L$-function.)
\subsubsection{Proof of Theorem (iii)}
Now we consider part (iii) of Theorem \ref{main}.  In this section, we study the constant terms of the Klingen Eisenstein series along the (unique up to conjugacy) proper parabolic subgroup $P$ of $U(r+1,1)$.
Recall that by lemma \ref{Harris} in the absolutely convergent range for $P$, the $A(\rho,f,s)_{-s}$-part in the expression for the constant term is $0$.
\begin{theorem}\label{thisprop-proposition}
Suppose $r=2$.  Then for $\phi\in\mathcal{X}^{pb}$, the following hold:
\begin{itemize}
\item[(i)] For all $\underline{k}_\phi,\kappa_\phi$ such that
 $L(\tilde{\pi}_\phi,\bar{\tau}_\phi^c,z_{\kappa_\phi}+1)$ is in the absolutely convergent range,
\begin{align*}
c_{(\underline{k}_\phi,0,\kappa_\phi)}=c_{(\underline{k}_\phi,0,\kappa_\phi)}',
\end{align*}
where $c_{(\underline{k}_\phi,0,\kappa_\phi)}$ and $c_{(\underline{k}_\phi,0,\kappa_\phi)}'$ are defined as in Definitions \ref{csubscriptdef1} and \ref{csubscriptdef2}, respectively.
\\
\item[(ii)] $F_\varphi(f_{\sieg,\phi}, z_{\kappa_\phi},g_0)=F_\varphi'(f_{\sieg,\phi}', z_{\kappa_\phi}+\frac{1}{2},g_0')\phi(\mathcal{L}_{\bar{\tau}'}^\Sigma)$, where $g_0=\prod_{v\not\in\Sigma}1\prod_{v\in\Sigma\backslash \{p,\infty\}}w\cdot\prod_{v=p}w_r$ and $g_0'=\prod_v1$.  (Here, we use the natural identification $(I(\rho_\infty)\otimes L^{(\underline{k})})^{K^{(r+1,1)}_\infty}=(V_\pi^{\smfin}\otimes L^{(\underline{k})})^{K^{(r,0)}_\infty}=(V_\pi^{\smfin}\otimes L^{(\underline{k},0;\kappa)})^{K^{(r,0)}_\infty}$ at $\infty$.)
\end{itemize}
\end{theorem}
\begin{proof}
Note that (i) only involves Archimedean computations.  To prove it, we use an auxiliary prime at which the form is ordinary. By Sato-Tate \cite{satotate1, satotate2}, we can find a prime $\ell$ split in $\mathcal{K}$ such that $\pi$ is ordinary at $\ell$. We can run the constructions of the $\ell$-adic $L$-functions and $\ell$-adic Klingen Eisenstein series as well. At the arithmetic points of scalar weights, we know (i) by the calculations in \cite[Lemma 4.1.2]{WAN}. Note that the set of arithmetic points of scalar weight is dense this implies the corresponding identity (ii) in the $\ell$-adic case (for the construction of $\ell$-adic Klingen Eisenstein series and $\ell$-adic $L$-functions). Now we specialize to vector-valued arithmetic points such that the corresponding special $L$-values are absolutely convergent (thus non-zero). This implies (i) for all such points.  Returning to our original $p$, then (i) implies (ii) for $p$ and $\pi$, which concludes the proof.
\end{proof}

To study the constant terms of the Klingen Eisenstein series we need to compare the period factors for Klingen Eisenstein series and $p$-adic $L$-functions. Note that \cite[Section 1]{SU06} explains that $(V_\pi\otimes V_{\underline{k}})^{K_\infty^{r,0}}=(V_\pi\otimes V_{(\underline{k},0;\kappa)})^{K_\infty^{r,0}}$, using minimal type theory (where we have identified $V_{\underline{k}}$ as a sub-representation of $V_{(\underline{k},0;\kappa)}|_{K_\infty^{(r,0)}}$). For $v\in I(\rho)\otimes V_{(\underline{k},0;\kappa)}^{K_\infty^{(r+1,1)}}$, we have $v(1)\in (V_\pi\otimes V_{\underline{k}})^{K_\infty^{(r,0)}}$, by considering the action of $K_\infty^{(r,0)}$. This implies that for the vector-valued Klingen Eisenstein series, the constant terms only have entries in $V_{\underline{k}}$. By the comparison of analytic and algebraic Fourier-Jacobi expansions in \cite[Section 2.8]{Hsieh}, the period factor showing up for the constant terms is $\Omega_\infty^{\frac{d_\phi}{2}}$. Therefore we have the following consequence:
\begin{corollary}\label{cor5.9}
Part (iii) of Theorem \ref{main} is true, i.e. if $r=2$ then the constant terms of $\mathbf{E}_{\varphi,\xi_0}$ are divisible by $\mathcal{L}_{\varphi,\xi_0}^\Sigma\cdot\mathcal{L}_{\bar{\tau}'}^\Sigma$.
\end{corollary}
This is a consequence of Proposition \ref{Harris}, Theorem \ref{thisprop-proposition}, and Lemma \ref{constant}, together with the calculations of the local pullback sections in the last section. Note that we are looking at the Fourier-Jacobi coefficient at $\beta=0$.

\section{Acknowledgement}

We would like to the thank the referee for a close reading of the paper and for helpful suggestions that improved the quality and readability of the paper.

\bibliography{SupersingularBibliography}

\textsc{Ellen Eischen, Department of Mathematics, The University of North Carolina at Chapel Hill, CB \#3250,
Chapel Hill, NC 27599-3250,
USA}\\
\indent \textit{E-mail Address}: eeischen@email.unc.edu\\

\textsc{Xin Wan, Department of Mathematics, Columbia University, New York, NY 10025, USA}\\
\indent \textit{E-mail Address}: xw2295@math.columbia.edu\\

\end{document}